\documentclass[reqno, 12pt, reqno]{amsart}

\usepackage{amsfonts, amsthm, amsmath, amssymb}

\usepackage{hyperref}

\RequirePackage{mathrsfs} \let\mathcal\mathscr

\numberwithin{equation}{section}

\newtheorem{theorem}{Theorem}[section]
\newtheorem{lemma}[theorem]{Lemma}

\newtheorem{corollary}[theorem]{Corollary}

\theoremstyle{definition}
\newtheorem*{ack}{Acknowledgements}
\newtheorem*{notation}{Notation}

\renewcommand{\d}{\mathrm{d}}
\renewcommand{\phi}{\varphi}

\newcommand{\0}{\mathbf{0}}

\newcommand{\PP}{\mathbb{P}}

\newcommand{\A}{\mathbf{A}}
\newcommand{\FF}{\mathbb{F}}
\newcommand{\ZZ}{\mathbb{Z}}

\newcommand{\NN}{\mathbb{N}}
\newcommand{\QQ}{\mathbb{Q}}
\newcommand{\RR}{\mathbb{R}}
\newcommand{\CC}{\mathbb{C}}

\newcommand{\cP}{\mathcal{P}}
\newcommand{\cN}{\mathcal{N}}

\renewcommand{\leq}{\leqslant}
\renewcommand{\le}{\leqslant}
\renewcommand{\geq}{\geqslant}
\renewcommand{\ge}{\geqslant}

\newcommand{\ma}{\mathbf}

\newcommand{\m}{\mathbf{m}}

\newcommand{\w}{\mathbf{w}}
\newcommand{\x}{\mathbf{x}}
\newcommand{\y}{\mathbf{y}}
\renewcommand{\c}{\mathbf{c}}
\renewcommand{\v}{\mathbf{v}}
\renewcommand{\u}{\mathbf{u}}
\newcommand{\z}{\mathbf{z}}
\renewcommand{\b}{\mathbf{b}}
\renewcommand{\a}{\mathbf{a}}

\renewcommand{\ss}{\mathfrak{S}}

\newcommand{\al}{\alpha}

\newcommand{\ve}{\varepsilon}

\newcommand{\ee}{\varepsilon}

\newcommand{\bth}{\boldsymbol{\theta}}
\newcommand{\bga}{\boldsymbol{\gamma}}

\DeclareMathOperator{\rank}{rank}

\DeclareMathOperator{\meas}{meas}
\DeclareMathOperator{\supp}{supp}

\DeclareMathOperator{\Mod}{mod} 
\DeclareMathOperator{\qorder}{ord_{\it{Q}}}
\renewcommand{\bmod}[1]{\,(\Mod{#1})}

\newcommand{\minor}{\mathfrak{m}}
\newcommand{\major}{\mathfrak{M}}

\renewcommand{\t}{\mathbf{t}}

\newcommand{\Xns}{X_{\mathrm{sm}}}

\renewcommand{\b}[1]{{\bf #1}}

\newcommand{\eeq}{\end{equation}}
\newcommand{\beql}[1]{\begin{equation}\label{#1}}

\begin{document}

\title[Intersections of cubic and quadric hypersurfaces]{Rational 
points on intersections of cubic and quadric hypersurfaces}

\author{T.D.\ Browning}
\author{R.\ Dietmann}
\author{D.R.\ Heath-Brown}

\address{School of Mathematics\\
University of Bristol\\ Bristol\\ BS8 1TW}
\email{t.d.browning@bristol.ac.uk}

\address{Department of Mathematics\\
  Royal Holloway\\ University of London\\
  Egham\\ TW20 OEX}
\email{rainer.dietmann@rhul.ac.uk}

\address{Mathematical Institute,
24--29 St. Giles', Oxford OX1 3LB}
\email{rhb@maths.ox.ac.uk}

\date{\today}


\begin{abstract}
We investigate the Hasse principle for complete intersections
cut out by a quadric and cubic hypersurface defined over the rational numbers.
\end{abstract}

\maketitle
\setcounter{tocdepth}{1}
\tableofcontents

\section{Introduction}

Suppose we are given a pair of forms $C,Q\in \QQ[x_1,\ldots,x_n]$,
with $C$ cubic and $Q$ quadratic,  whose common zero locus defines a
complete intersection 
 $X\subset\PP^{n-1}$ defined over $\QQ$.
The primary goal of this paper is to establish the existence of $\QQ$-rational points on 
$X$ 
under the mildest possible  hypotheses.  

One of the few  results in the literature that specifically treats
pairs of cubic and quadratic forms appears in  work of Wooley
\cite{wooley', wooley}. This deals with the special case in which  $C$
and $Q$ are both diagonal, so that  
$$
C=a_1x_1^3+\cdots + a_nx_n^3, \quad 
Q=b_1x_1^2+\cdots + b_nx_n^2,
$$
for integers $a_i,b_i$, with the $b_i$ not all sharing the same sign.
Assuming that $n\geq 13$,  it follows from the main result in
\cite{wooley} that  
$X(\QQ)$ is non-empty provided only that $X(\RR)\neq \emptyset$, 
with 
at least seven $a_i$ non-zero.

In our work we wish to handle general forms $C,Q$ in so far as is possible. 
All of the results that we obtain pertain to complete intersections
$X\subset \PP^{n-1}$ cut out by  
a cubic  hypersurface $C=0$ and a quadric hypersurface
$Q=0$, both defined over $\QQ$.

One way to produce rational points on $X$ is 
first to find a large dimensional linear space on the quadric $Q=0$,
which is defined over $\QQ$.  
One is then led to the simpler problem of finding rational points on
the intersection of the cubic hypersurface $C=0$ with the  linear
space.  
Let us call a $d$-dimensional 
linear space $\Lambda\subset \PP^{n-1}$ a {\em $d$-plane}.  
Let  $Q\in \QQ[x_1,\ldots, x_n]$ be a quadratic form.
For each prime $p$
the quadric $Q=0$ contains a
$\QQ_p$-rational $d$-plane 
providing that 
$$
n \ge 5+2d.
$$ 
The case $d=0$ corresponds to the well-known fact 
that every quadratic form in at least five variables is isotropic over
$\QQ_p$. The general case follows from inserting this fact into 
work of Leep \cite[Corollary~2.4~(ii)]{leep}. 
Moreover, the quadric $Q=0$ 
contains a real $d$-plane, provided that 
$d\leq n-1-\max(r,s)$, 
where $(r,s)$ is 
the signature of $Q$. 
The existence of a $d$-plane in the quadric everywhere locally 
is enough to ensure the existence of a
$\QQ$-rational $d$-plane $\Lambda$ contained in the quadric,
by the Hasse principle for linear spaces on quadratic
forms (see the proof of \cite[Theorem 2]{BDLW}, for example).
As soon as $d\geq 13$ we may apply the main result in work of
Heath-Brown \cite{14}, which shows  
that $C=0$ has a rational point on $\Lambda$, giving a rational
point on $X$. 
Finally, it is clear that we may take $d=13$ whenever
 $n \ge 31$ and 
$n-\max(r, s) \ge 14$. We record this observation as follows. 

\begin{theorem}
\label{elementary1}
Suppose that $n\geq 31$ and $Q$ has 
signature $(r,s)$, with  $\max( r, s) \le n-14$. Then $X(\QQ)\neq \emptyset$.
\end{theorem}

It is worthwhile noting that when  working over totally imaginary
number fields $k$,  the assumption on the signature of the quadratic
form can be removed.  
Appealing to work of Pleasants \cite{P}, which is valid for cubic
forms in at least $16$ variables over any number field,  
one concludes that $X(k)\neq \emptyset$ provided only that $n\geq 5+2(16-1)=35$.

Our next results are established using the Hardy--Littlewood circle
method directly.  We write  $\Xns$ for  the smooth locus  of
 points on $X$.   
Recall that the smooth
Hasse principle is said to hold for a family of such varieties when
the existence of a point in  
$\Xns(\A)=\Xns(\RR)\times \prod_p \Xns(\QQ_p)$,
where $\A$ denotes the ad\`eles, is enough to ensure the existence of
a smooth 
$\QQ$-rational point in $X$. 
Given a form $F\in \QQ[x_1,\ldots, x_n]$,   
we define the {\em $h$-invariant} $h(F)$ to be the least positive integer
$h$ such that the  
$F$ can be written identically as 
$$
A_1B_1 +\cdots+A_hB_h,
$$
for forms $A_i,B_i\in\QQ[x_1,\ldots,x_n]$ of positive degree. 
Taking $R=5$,  $r_3=r_2=1$ and $k=3$ in work of Schmidt \cite[Theorem
II]{schmidt}, we obtain the smooth Hasse principle for $X$ provided
that $h(C)\geq 480$ and $h(Q)\geq 30$. We note here that one 
clearly has $\rank(Q)\le 2h(Q)$ for any quadratic form, so that it 
suffices to have $h(C)\geq 480$ and $\rank(Q)\geq 59$. 
With this in mind we state the following result.

\begin{theorem}\label{main'} 
Write $\rank(Q)=\rho$.  Then the smooth  Hasse principle 
holds for $X$ provided that 
$$
(h(C)-32)(\rho-4)>128.
$$
In particular it suffices to have 
$\min( h(C),\rho)\ge 37$.

If $C$ is non-singular then the 
smooth  Hasse principle holds for $X$ provided that 
$$
(n-32)(\rho-4)>128.
$$
\end{theorem}

There is an old result of  Birch \cite{birch} which establishes the
smooth Hasse principle for complete intersections $V\subset \PP^{n-1}$
cut out by  forms  
$F_1,\ldots,F_R$
of equal degree $d$, provided that the inequality
$$ 
n-\dim V^*> R(R+1)(d-1)2^{d-1}
$$
holds, 
where $V^*$ is the affine variety cut out by the 
condition 
\[\rank (\nabla F_i)_{1\leq i\leq R}<R.\] 
It is not entirely clear how this method could be adapted to handle 
a system of forms of differing degree, since the process of Weyl
differencing involved in the proof eradicates the presence of the
lower degree forms.  A satisfactory treatment of this issue is a key
ingredient in Theorem \ref{main'}. 
Schmidt encounters the same problem in the work \cite{schmidt} cited
above, and deals with it in a simpler but less effective manner.  When
the exponential sums involved have only one variable the ``final
coefficient lemma'' (see Baker \cite[Section 4.2]{bb}) gives very good
results.  However this relies ultimately on the use of strong
bounds for complete exponential sums, which are not available when one has 
several variables.

When $X$ is assumed to be non-singular we will show in Corollary~\ref{c:3.2}
that the cubic form $C$ can be taken to be non-singular with the
quadratic form $Q$ having  
rank $\rho\geq n-1$. Theorem \ref{main'} therefore implies that the
Hasse principle holds for  
non-singular $X$ provided that $n\geq 37$. 
The following result improves on this further.

\begin{theorem}\label{main}
Suppose that $X$ is non-singular, with $n \geq 29$.
Then $X(\QQ)\not=\emptyset$ if and only if $X(\RR)\not=\emptyset$.
\end{theorem}

Theorem \ref{main} establishes the Hasse principle for non-singular
$X$, with $n\geq 29$. 
The issue of determining when $X$ has $p$-adic
points for every prime $p$ 
is of considerable interest in its own right. Artin's conjecture would
imply that it is sufficient to have $n>3^2+2^2=13$. 
Indeed it has been shown by Zahid \cite{zahid} that an arbitrary 
intersection $X: C=Q=0$ with $n>13$ has $X(\QQ_p)\neq \emptyset$ for 
every prime $p>293$. 
However if $n\ge 29$ we can in fact  
recycle the proof of Theorem \ref{elementary1} to deduce that 
the quadric hypersurface
$Q=0$ contains 
a $\QQ_p$-rational projective space of dimension at least 
$ \lceil(\rho-6)/2\rceil$. 
The existence of a point in $X(\QQ_p)$ is then assured by 
an old result of Lewis \cite{lewis}, which shows that the cubic $C=0$
has a $\QQ_p$-rational point on any $\QQ_p$-rational 
projective linear space of dimension $9$ or more. 

Our proof of Theorems \ref{main'} and \ref{main} is based on the
Hardy--Littlewood circle 
method.  We will give an overview of the proof in Section \ref{s:2}. 
As is usual with the circle method our arguments show not only that 
$X(\QQ)$ is non-empty, but may even be developed to  
establish weak approximation. 
Moreover, we can prove a variant of Theorem \ref{main} which applies
to singular $X$.  Suppose that $\sigma\geq -1$ is the dimension of the
singular locus of $X$, with the
convention that $\sigma=-1$ if and only if $X$ is
non-singular. Then an argument based on Bertini's theorem can be
used to show that the smooth Hasse principle holds for   
$X$, provided that $n\geq 30+\sigma$. 
We leave the details of both of these remarks to the reader.  

To state our remaining  result, we need to introduce some more terminology.
If $F \in K[x_1, \ldots, x_n]$ for some field $K$, then
we define the \emph{order}
of $F$ to be the minimal non-negative integer $m$ such that there
exists a 
matrix $\mathbf{T}\in \mathrm{GL}_n(K)$ 
with the property that in
$
  F(\mathbf{T}(x_1, \ldots, x_n))
$
only $m$ of the variables $x_1, \ldots, x_n$
occur with a non-zero coefficient.
It is a familiar fact that the order of $F$ is independent of the 
field of definition $K$. 
If $Q \in \QQ[x_1, \ldots, x_n]$ is a quadratic form, then we call a
pair of  cubic  
forms $C_1, C_2 \in \QQ[x_1, \ldots, x_n]$  \emph{$Q$-equivalent}
if there exists a linear form $L \in \QQ[x_1, \ldots, x_n]$ such that
$$
C_1-C_2=LQ.
$$ 
It is easily checked that this indeed defines an
equivalence relation on the set of rational cubic forms, and that the
set of zeros of the intersection $C=Q=0$ does not change if one replaces
$C$ by another cubic form that is $Q$-equivalent to $C$.
Finally, for a fixed quadratic form $Q \in \QQ[x_1, \ldots, x_n]$ and 
cubic form $C_1 \in \QQ[x_1, \ldots, x_n]$, we define the
\emph{$Q$-order} $\qorder(C_1)$ of $C_1$ to be the minimal order
amongst all cubic forms $C_2$ that are $Q$-equivalent to $C_1$.
We are now ready to reveal the following result.

\begin{theorem}
\label{small_h}
Suppose that $n\ge 49$ 
and $\qorder(C) \ge 17$, and that $\Xns(\RR)\neq \emptyset$.  
Then $X(\QQ) \ne \emptyset$. 
\end{theorem}

The hypothesis 
$\qorder(C) \ge 17$ in the previous theorem
can be weakened to $\qorder(C) \ge 14$, provided that we 
impose the additional assumption that 
for any cubic form that is $Q$-equivalent to $C$, if the 
corresponding cubic hypersurface has rational points then 
they are dense in the locus of real points. 

Simple considerations show that Theorem \ref{small_h} 
could not be true without
some sort of assumption on the $Q$-order of $C$.
We assume that $n\geq 49$, in order to fall within the range of the theorem.
Let $m\leq n$ and suppose that $C\in \QQ[x_1,\ldots,x_m]$ is a cubic
form for which  
$C=0$ has no $\QQ$-rational point. 
In particular  $C$ must be non-degenerate, so that $m$ is the order of $C$.
Let $X$ be the variety cut out by $C$ and the 
quadratic form
\[
  Q(x_1, \ldots, x_n) = -x_m^2+x_{m+1}^2+\cdots +x_n^2.
\]
It is clear that $\Xns(\RR)\neq \emptyset$
and  $\qorder(C)= m$.
Any rational point on $X$ would lie on $C=0$, so that
$x_1=\cdots=x_m=0$. 
Then $Q=0$ implies that $x_{m+1}=\cdots=x_n=0$, whence in 
fact $X(\QQ)=\emptyset$. 
This example shows that if one had a version of 
Theorem \ref{small_h} in which the condition on the $Q$-order of $C$ 
were relaxed to $\qorder(C) \ge 13$ then we would be able to 
deduce that any cubic over $\QQ$ in 13 variables has a non-trivial 
rational zero.  In particular any such improvement of Theorem~\ref{small_h} 
would lead to a corresponding sharpening of the result 
of \cite{14}. 

Mordell \cite{mordell} has  constructed a non-degenerate cubic form
$C$ in $9$ variables 
for which $C=0$ has no $\QQ_p$-point for some prime $p$, 
and hence has no point over $\QQ$. 
This shows that, aside from extending the range for $n$, the best one
can hope for 
in Theorem \ref{small_h} 
is a reduction of the lower bound on the $Q$-order of $C$ to
$\qorder(C)\geq 10$. 
Moreover, our example
shows that it is really the $Q$-order of $C$ that matters rather than the
order, since we could replace $C$ by $C+LQ$ for a linear form $L$ 
and in this way increase the order of the 
cubic form.

\begin{notation}
Throughout our work $\NN$ will denote the set of positive
integers.  For any $\al\in \RR$, we will follow common convention and
write $e(\al):=e^{2\pi i\al}$ and $e_q(\al):=e^{2\pi i\al/q}$. 
The
parameter $\ve$ will always denote a small positive real
number.
We shall use $|\x|$ to denote the norm $\max |x_i|$ 
of a vector $\x=(x_1,\ldots,x_n)$. 
All of the implied constants that appear in this 
work will be allowed to
depend upon the coefficients of the forms $C$ and $Q$ under
consideration, the number $n$ of variables  involved, and the
parameter $\ve>0$.  Any further dependence will be explicitly
indicated by appropriate subscripts.
\end{notation}

\begin{ack}
Most of this work was carried out 
during the programme ``Arithmetic and geometry'' at the {\em Hausdorff
  Institute} in Bonn,  
for whose hospitality the authors are very grateful. 
While working on this paper the first  author was
supported by ERC grant \texttt{306457}. 
\end{ack}

\section{Overview of the paper}\label{s:2}

We have already established Theorem \ref{elementary1}. 
In Section \ref{s:3} we will collect together some geometric facts
that will be used in the proof of Theorems
\ref{main'}--\ref{small_h}. 
Theorems \ref{main'} and \ref{main} will be established using the
Hardy--Littlewood circle method. 
This will occupy the bulk of our paper (Sections \ref{s:4}--\ref{s:5}). 
Finally, in Sections \ref{s:10} and \ref{s:11},  we will turn to the
proof of Theorem \ref{small_h}.

The aim of the present section is to survey
the key ideas in the proof of Theorems \ref{main'} and \ref{main}. 
On multiplying through by a common denominator we can ensure that $C$
and $Q$ have coefficients in $\ZZ.$ 
In both results the goal will be to establish an asymptotic formula
for the quantity 
\begin{equation}\label{eq:def-N}
N_\omega(X;P):= \sum_{\substack{\x\in \ZZ^n\\ C(\x)=Q(\x)=0}} 
\omega(\x/P),  
\end{equation}
as $P\rightarrow \infty$, for a suitably chosen function $\omega:\RR^n
\rightarrow \RR_{\geq 0}$ with support in $(-1/2,1/2)^n$.  All of our
weight functions will be infinitely differentiable, with  bounded
Sobolev norms.  
The starting point in the circle method 
is the  identity
$$
N_\omega(X;P)=\int_{0}^1\int_0^1 S(\al_3,\alpha_2) \d\alpha_3\d \alpha_2,  
$$
where 
\begin{equation}
  \label{eq:S}
S(\al_3,\al_2):=\sum_{\x\in\ZZ^n}\omega(\x/P) e\left(\al_3 C(\x)+\al_2 Q(\x)\right),
\end{equation}
for any $\al_3, \al_2\in \RR$. 
The idea is then to divide the region $[0,1]^2$ into a
set of major arcs $\major$ and minor arcs $\minor$. 
In the usual way we seek to prove an asymptotic formula 
\begin{equation}\label{eq:major}
\iint_{\major} S(\al_3,\al_2) \d\al_3\d \al_2\sim c_X P^{n-5},
\end{equation}
as $P\rightarrow \infty$, together with a satisfactory bound on the minor arcs
\begin{equation}\label{eq:minor}
\iint_{\minor} S(\al_3,\al_2) \d\al_3\d \al_2  
=o( P^{n-5}).
\end{equation}
Here the constant $c_X$ will be a product of local densities, which is
positive when $\Xns(\A)$ is non-empty.

For any pair $\al_3,\al_2$ we will produce a simultaneous rational
approximation $a_3/q, a_2/q$ using a two dimensional version of
Dirichlet's approximation theorem.  To describe this we take positive
integers $Q_3,Q_2$ satisfying
\begin{equation}\label{eq:Q2Q3}
Q_3:= [P^{4/3}]\;\;\;\mbox{and}\;\;\; Q_2:= [P^{1/3}].
\end{equation}
Then, by the pigeon hole principle, there will be 
$\a=(a_3,a_2)\in \ZZ^2$ and $q\in \NN$
such that $q\le Q_3Q_2$ and $\gcd(q,\a)=1$, for which
\begin{equation}\label{eq:al2al3}
\left|\al_3-\frac{a_3}{q}\right|\le\frac{1}{qQ_3},\;\;\;
\mbox{and}\;\;\; \left|\al_2-\frac{a_2}{q}\right|\le\frac{1}{qQ_2}.
\end{equation}
It will therefore be convenient to write
$$
\al_3=\frac{a_3}{q}+\theta_3\;\;\;\mbox{and}\;\;\;\al_2=\frac{a_2}{q}+\theta_2.
$$

Let $\delta\in(0,1/3)$ be a parameter to be decided upon later (see
\eqref{eq:choosedelta}). 
We will take as major arcs
$$
\major:=\bigcup_{q\leq P^{\delta}}\bigcup_{\substack{\a\bmod{q}\\
\gcd(q,\a)=1}}\major_{q,\a},
$$
where
$$
\major_{q,\a}:=\left\{(\alpha_3,\alpha_2)\bmod{1}: 
\left|\alpha_i-\frac{a_i}{q}\right|\leq
P^{-i+\delta}, \mbox{ for $i=3,2$}\right\}. 
$$
It is easy to see that $\major_{q,\a}\cap \major_{q',\a'}=\emptyset$
whenever  $\a/q\neq \a'/q'$, provided that $P$ is taken to be
sufficiently large. 
Moreover each major arc is contained in the corresponding range given
by \eqref{eq:al2al3}.

Our treatment of \eqref{eq:major}
is relatively
standard and is the focus of Section \ref{s:5}.

The  minor arcs are defined to be $\minor=[0,1]^2\setminus \major$.
Thus they are defined by having either $q>P^{\delta}$ or
$\max(|\theta_3|P^3\,,\,|\theta_2|P^2)>P^{\delta}$. 
Our estimation of $S(\alpha_3,\alpha_2)$ for $(\alpha_3,\al_2)\in \minor$
will differ according to the hypotheses placed on $X$. A common
ingredient will be a more efficient  version of Weyl differencing,
which draws inspiration from the work of Birch \cite{birch}, but which
is specially adapted to systems of equations of differing degree. 
Suppose that
$$C(x_1,\ldots,x_n)=\sum_{i,j,k=1}^n c_{ijk}x_ix_jx_k,
$$
for integer coefficients $c_{ijk}$ that are symmetric in the indices
$i,j,k$.  Define the bilinear 
forms 
$$
B_i(\x;\y):=3!\sum_{j,k=1}^n c_{ijk}x_j y_{k},  \quad (1\leq i\leq n).
$$
Using two successive applications of Weyl
differencing,  as in Birch's work, 
we can relate the size of the exponential sum
$S(\al_3,\al_2)$ to the locus of integral points on the affine variety
given by the simultaneous equations
$B_i(\x;\y)=0$, for $1\leq i\leq n$.  When  $C$ defines a smooth cubic
hypersurface, or when $h(C)$ is sufficiently
large, we shall be able to get good estimates for $S(\alpha_3,\al_2)$
unless $\alpha_3$ happens to be close to a rational number with small
denominator.  If this occurs then we shall use a single Weyl squaring,
modified in a way motivated by van der Corput's method so as to remove
the effect of the cubic terms. 
This step marks a departure from the approach of Birch, which is
completely insensitive to  
the quadratic form $Q$ that appears in the sum. 
Our modified version of Weyl differencing is the subject of Section
\ref{s:4}, and is one of the more novel parts of the paper. 
The work in this section will ultimately suffice to
establish Theorem \ref{main'} in Section \ref{s:4a}.  

When it comes to establishing Theorem \ref{main}, for which $X$ is
assumed to be non-singular, the work in Section \ref{s:4a}  only allows us
to establish an asymptotic formula for $N_\omega (X;P)$ when $n\geq
37.$ 
Instead, in Section \ref{s:7}, we shall produce a companion estimate for
$S(\alpha_3,\alpha_2)$, which is based on Poisson summation. Once
combined with the work in Section \ref{s:4}, this will lead to  
an asymptotic formula for $N_\omega (X;P)$ when $n\geq 29$, 
as required for Theorem \ref{main}.  One inconvenient feature of this
combined attack is that, while both methods involve rational
approximations to $\al_3$ and $\al_2$, there is no {\em a priori}
guarantee that the rational approximations occurring in the two
methods are the same.

\section{Geometric preliminaries}\label{s:3}

Let $k$ be a field of characteristic zero. 
Suppose $V\subset \PP^{n-1}$ is a non-singular complete intersection
of codimension $r$,  
whose homogeneous ideal in $k[\x]=k[x_1,\ldots,x_n]$ is generated by $r$ 
forms $F_1,\ldots,F_r\in k[\x]$.  
Suppose that the maximum degree attained by any form is attained by $F_1$.
One has a great deal of freedom in the choice of $F_1$, since one may
equally take $F_1+\sum_{1<i\leq r} H_{i} F_i$ for any forms $H_i\in
k[\x]$ such that  $\deg H_iF_i=\deg F_1$. 
In this way it is reasonable to expect that one can always arrange for
the leading form $F_1$ to be non-singular, provided that $V$ itself is
non-singular.  
This is made precise in the following result due to  Aznar \cite[\S 2]{aznar}.

\begin{lemma}\label{lem:aznar}
Let  $V\subset \PP^{n-1}$ be a non-singular complete intersection  of
codimension $r$, which is defined over a field $k$ of characteristic
zero.   
Then there is a system of generators $F_1,\ldots,F_r\in k[\x]$ of the
ideal of $V$,  
with 
$$
\deg F_1\geq \cdots \geq \deg F_r,
$$ 
such that the varieties
$$
W_i: \quad F_1=\cdots=F_i=0, \quad (i\leq r),
$$
are all non-singular.
\end{lemma}

\begin{proof}
To be precise Aznar works with $k=\CC$, but the adaptation to
arbitrary fields of non-zero characteristic is straightforward. We
give the proof here for the sake of completeness.  
We argue by induction on $i$, the case $i=0$ being trivial.  

Now let $i$ be such that $1\le i\leq r$.  Fix a system of generators
$$
F_1,\ldots,F_{i-1}, G_i,\ldots,G_r\in k[\x]
$$  
for the ideal of $V$, 
with 
$$
\deg F_1\geq \cdots \geq \deg F_{i-1}\geq \deg G_i \geq \cdots \geq \deg G_r, 
$$ 
such that the varieties $W_{1}, \ldots,W_{i-1}\subseteq \PP^{n-1}$ are
all non-singular. 
Suppose that $d_k=\deg G_k$, for $i\leq k\leq r$. 
Let us write 
$$
f_0=G_i, \quad f_{j,k}=x_j^{d_i-d_{k}}G_{k}, 
$$
for $1\leq j\leq n$ and $i< k\leq r$. This gives a system 
$$\mathbf{f}=(f_0,f_{1,i+1}, \ldots, f_{n,r})$$ of 
$N=1+n(r-i)$
forms in 
$k[\x]$
of degree $d_i$.  The set of points 
in $W_{i-1}$  for which $\mathbf{f}(\x)=\mathbf{0}$ 
 precisely coincides with the non-singular variety $V$. We let 
$U=W_{i-1}\setminus V$.
Consider the morphism  
$$
\pi: U\rightarrow \PP^{N-1},
$$ 
given by $[\x]\mapsto [\mathbf{f}(\x)]$.
Then an application of Bertini's theorem (see Harris \cite[Theorem
17.6]{harris}, for example) 
reveals that for a general hyperplane $H\subset \PP^{N-1}$ the fibre 
$\pi^{-1}(H)$ is non-singular.  This means that for a general choice
of $\lambda_0, \lambda_{j,k}\in k$, the  
degree $d_i$ form 
$$
F_i=\lambda_0 G_i +\sum_{\substack{1\leq j\leq n\\i< k\leq r}}
\lambda_{j,k} x_j^{d_i-d_{k}}G_{k} 
$$
is defined over $k$  and 
$U\cap \{F_i=0\}$ is non-singular. 
This implies that 
$$
W_i: \quad F_1=\cdots=F_i=0
$$
is non-singular, since $V$ is non-singular. 
The induction hypothesis therefore follows, which  completes the proof
of the lemma.  
\end{proof}

We apply this result to the complete intersection in Theorem
\ref{main} to deduce the following consequence.  

\begin{corollary}\label{c:3.2}
Let $X\subset \PP^{n-1}$ be a non-singular complete intersection, 
cut out by a cubic and quadric hypersurface defined over $\QQ$. 
Then there exists a non-singular cubic form $C\in \ZZ[\x]$ and a
diagonal quadratic form $Q\in \ZZ[\x]$ of rank at least $n-1$, such
that $X$ is given by  $C=Q=0$. 
\end{corollary}

\begin{proof}
Taking $k=\QQ$ in Lemma \ref{lem:aznar} ensures the 
existence of a non-singular cubic form $C\in \QQ[\x]$ and a quadratic
form $Q\in \QQ[\x]$  
such that $X$ is given by $C=Q=0$. 
After a non-singular rational change of variables we may further
assume that $Q$ is diagonal.  
By multiplying through by a common denominator we can ensure that $C$
and $Q$ are both defined over $\ZZ$. 

Showing that  $\rank (Q)\geq n-1$ is equivalent to showing that the
quadric hypersurface $Q=0$ in $\PP^{n-1}$ must have  singular locus
of dimension less than $1$. But if the singular locus had positive
dimension its intersection with the  
cubic hypersurface $C=0$ would be non-empty and every point in it
would be a singular point of $X$. This contradicts the non-singularity
of $X$, which thereby completes the proof. 
\end{proof}

One of the hallmarks of Theorem \ref{small_h} is that it applies to
very general  
complete intersections $X\subset \PP^{n-1}$ cut out by a cubic
hypersurface $C=0$ and a quadric hypersurface $Q=0$.  
Let us define $h_Q(C)$ to be the minimal value of 
$h(C+LQ)$ as $L$ varies over all linear forms defined over $\QQ$.   
We
remark at once that $\qorder(C)\geq h_Q(C)$ and  
\begin{equation}\label{h-hq} 
h_Q(C)\le h(C)\le h_Q(C)+1.
\end{equation}  
We will require easily checked criteria  on the defining forms which
are sufficient to ensure that $X$ is absolutely irreducible. This is
the purpose of the following result.

\begin{lemma}\label{lem:irred}
Let $X\subset \PP^{n-1}$ be a variety
cut out by a cubic hypersurface $C=0$ and a quadric hypersurface
$Q=0$, both defined over $\QQ$. Assume that  
$\rank(Q)\geq 5$, that $\qorder(C)\ge 4$ and that $h_Q(C)\ge 2$. Then $X$ is
an absolutely irreducible variety of codimension $2$ and degree $6$.
\end{lemma}

We begin by showing that the lemma applies under the hypotheses of Theorem
\ref{main'}.  The condition $\rank(Q)\geq 5$ is automatically met.
For the first part of the theorem, which requires $h(C)\geq 33$, the
remaining conditions 
of Lemma \ref{lem:irred} are clearly met since $\qorder(C)\geq
h_Q(C)\geq 32$, by \eqref{h-hq}. 
For the second part of the theorem, which requires $C$ to be
non-singular and $n\geq 33$,  
we claim that $\qorder(C)\geq 4$ and $h_Q(C)\geq 2$. Indeed, if
$h_Q(C)=1$ then $C$ takes the shape $L_1Q+L_2Q_2$ for suitable linear
forms $L_1,L_2$ and a quadratic form $Q_2$, all defined over $\QQ$. 
Since $n\geq 33$ the intersection $L_1=L_2=Q=Q_2=0$ is non-empty and produces
a singular point of $C=0.$
Alternatively, if $\qorder(C)\leq 3$ then we could take $C$ to have 
the shape $C_1(x_1,x_2,x_3)+LQ$, for a suitable linear  
form $L$ and a suitable cubic form $C_1$, both defined over
$\QQ$. Again, since $n\geq 33$ we could find a singular point of $C=0$
by considering the intersection $x_1=x_2=x_3=L=Q=0$. 
This shows that the $X$ considered in Theorem \ref{main'} are
absolutely irreducible under the hypotheses presented there. 
  
For Theorem \ref{main} we see from Corollary \ref{c:3.2}
that we will have 
\[\rank(Q)\ge n-1\ge 28>5.  \]
Moreover a variety 
$Q=L'Q'=0$ will have singular points
wherever $Q=L'=Q'=0$.  Thus if $X$ is non-singular we must have
$h_Q(C)\ge 2$.  Similarly a variety
$Q(x_1,\ldots,x_n)=C'(x_1,x_2,x_3)=0$ will have singular points
wherever $Q(0,0,0,x_4,\ldots,x_n)=0$, so that if $X$ is non-singular we will
have $\qorder(C)\ge 4$. It follows that the lemma applies for Theorem
\ref{main}.  Finally, for Theorem \ref{small_h}, the lemma will apply
unless $h_Q(C)\le 1$ or $\rank(Q)\leq 4$. 

\begin{proof}[Proof of Lemma \ref{lem:irred}] 
Under the hypotheses of the lemma, the forms $C$ and $Q$ share no
common factor of positive degree. 
Hence $X$ is pure dimensional.  
Suppose that $X$ decomposes into irreducible components $Z_1\cup
\cdots \cup Z_t$. 
It follows from  B\'ezout's theorem (in the form given by
\cite[Example~8.4.6]{fulton}) 
that 
$$
\deg(Z_1)+\cdots +\deg(Z_t)\leq 6.
$$
Each 
$Z_i$ is an irreducible codimension $1$ divisor on the quadric
hypersurface $Q=0$. 
Let $Z$ be one of these components. 
Since $\rank(Q)\geq 5$, by hypothesis, it follows from Klein's theorem
(see Hartshorne \cite[Part II, Ex. 6.5(d)]{hart}) that 
there is an irreducible hypersurface $W\subset \PP^{n-1}$ such that 
$Z$ is the intersection of $W$ with the quadric $Q=0$, with multiplicity $1$. 
But then a further application of B\'ezout's theorem (see \cite[\S
8.4]{fulton}) implies that $\deg(Z)$ must be even. 

In order to conclude the proof of the lemma it clearly  suffices to 
show that $Z$ cannot have degree $2$. 
Suppose, for a contradiction, that $Z$ is quadratic. Then Klein's
theorem shows that $Z$ is given
by $L=Q=0$, say, where $L$ is a linear form defined over
$\overline{\QQ}$. 
It follows that $C$ must take the shape $LR+\tilde L Q$, where $\tilde L$ and $R$ 
are linear and quadratic forms respectively, defined over
$\overline{\QQ}$. Indeed if $k$ is the minimal field of definition for $L=0$
then we may choose $R$ and $\tilde L$ in such a way that they too are
defined over $k$.  Thus if $k=\QQ$ we will have $h_Q(C)\le 1$, contrary
to assumption. If $k$ is a quadratic extension of $\QQ$ then $Z$ and
its quadratic conjugate will be distinct components of $X$, and there
will therefore be a third component of degree 2, which must be defined
over $\QQ$.  We may then deduce as above that $h_Q(C)\le 1$. We cannot
have $[k:\QQ]>3$ since the number of components $Z_i$ is at most 3, so
that we are left with the case in which $k$ is cubic.

Let $L=L_1,L_2$ and $L_3$ be the three conjugates of $L$, 
and write $C=L_iR_i+\tilde L_i Q$ accordingly.  Thus
$LR+\tilde L Q=L_2R_2+\tilde L_2 Q$, so that $LR=0$ whenever
$L_2=Q=0$.  However the variety $L_2=Q=0$ is absolutely irreducible,
since $\rank(Q)\ge 5$, and it follows that one or other of $L$ and
$R$ must vanish whenever $L_2=Q=0$. The only hyperplane
containing $L_2=Q=0$ is the obvious one $L_2=0$, so in the first case
$L$ and $L_2$ must be proportional.  This however is impossible,
since we have eliminated the case in which the hyperplane $L=0$ is
defined over $\QQ$.  Thus $R$ must vanish on $L_2=Q=0$, so that
$R= L_2L_2'+c_{2}Q$ for some linear form $L_2'$ and constant
$c_2$, both defined over $\overline{\QQ}$.

In the same way we will have $R= L_3L_3'+c_3Q$, say.
Then
\[(c_2-c_3)Q=(R-L_2L_2')-(R-L_3L_3')=L_3L_3'-L_2L_2'.\] 
Since $\rank(Q)\ge 5$ this can happen only when $c_2=c_3$.  We
will write $c=c_2=c_3$ for this common value.  We
then have $L_3L_3'=L_2L_2'$, and since $L_2$ and $L_3$ are not
proportional, by the argument above, we see that $L_3'=\gamma L_2$ 
for some constant $\gamma$. Thus $R=\gamma
L_2L_3+c Q$, so that $C=\gamma L_1L_2L_3+(cL_1+\tilde L_1) Q$. 

We may now write $C=\gamma N+M Q$ 
where $N=L_1L_2L_3$ is a cubic
norm form, defined over $\QQ$, and $\gamma$ and $M$ 
are a constant
and a linear form respectively, both over $\overline{\QQ}$. Since $N$
and $Q$ have no common factor this representation must be unique, so
that in fact $\gamma$ and $M$ are defined over $\QQ$.  We then
deduce that $\qorder(C)\le\qorder(\gamma N)\le 3$, contrary to
our hypotheses.  The lemma therefore follows.
\end{proof}

To deal with the local solubility conditions in Theorem \ref{small_h}, 
we will also need some information  about varieties over local fields. 
The following fact is certainly well-known (see Koll\'ar \cite[\S
2.3]{kollar}, for example), but  we recall the proof here for
completeness.  

\begin{lemma}\label{lem:density}
Let $k$ be  $\RR$ or a finite extension of a $p$-adic field $\QQ_p$. 
Let $V$ be an absolutely irreducible projective variety defined
over $k$ with a smooth $k$-point. Then 
$V(k)$ is  dense in $V$ under the Zariski topology.
\end{lemma}

\begin{proof}
Suppose we are given a smooth point $x\in V(k)$, but that $V(k)$ is not 
Zariski-dense in $V$. Then one may find a non-singular curve $C$ in
$V$ which passes through    
$x$ and which only contains finitely many $k$-points. 
There is a non-constant rational map 
$$
C \rightarrow \PP^1,
$$ 
which is unramified at $x$.  
As this map is unramified, the differential at $x$
is an isomorphism, and therefore, by the inverse function theorem
(see Serre \cite[Part II, \S III.9]{serre}, for example) the induced map
$C(k) \rightarrow \PP^1(k)$
is an isomorphism of analytic manifolds in a neighbourhood of $x$ (in the  
topology induced by the topology of $k$). 
Now $\mathbb{P}^1(k)$ has infinitely many $k$-points in any
neighbourhood of any point, so by lifting such points to $C(k)$ 
by the inverse local isomorphism we find infinitely many $k$-points
on $C$, which is a contradiction.
\end{proof}

\section{Weyl differencing}\label{s:4}

In this section we will use the Weyl differencing approach to give bounds for
$S(\al_3,\al_2)$, defined in \eqref{eq:S}. Our overall strategy will
be to assume that $S(\al_3,\al_2)$ is large, and to deduce that
$\al_3$ has a good approximation by a rational number with small denominator.
Using this information we then go on to show that $\al_2$ must also have
a good approximation by a rational number with small denominator.  The
first phase of the argument will apply Weyl's method to
$|S(\al_3,\al_2)|^4$.  In contrast the second phase will use
$|S(\al_3,\al_2)|^2$, and will incorporate an idea related to van der
Corput's method.  The reader will see that in the first stage it is
only the cubic form $C(\x)$ which is relevant, while in the second
stage it is primarily the quadratic form $Q(\x)$ which features.

For the first phase of the argument we write $h=n$ if the form $C$ is
non-singular, and otherwise take $h=h(C)$.  Notice that for
Theorems~\ref{main'} and \ref{main} we must have 
$h\ge 29$, as we henceforth assume.  We now define 
$T_3=T_3(\al_3,\al_2)\in\RR_{>0}\cup\{\infty\}$ by setting
\beql{eq:T3def}
|S(\al_3,\al_2)|=P^nT_3^{-h}.
\eeq
We then call on Lemma 1 of Davenport and Lewis \cite{DL}. We will
require a version with some trivial modifications, as we will
explain.  Let $R>1$ and define
\[n(R):=\#\{(\x,\b{y})\in\ZZ^{2n}:|\x|<R,\,|\b{y}|<R,\;
B_i(\x;\b{y})=0\;\forall i\le n\}.\]
Then if $\ee>0$ is given, the lemma, suitably modified, shows that either 
\beql{eq:nRc}
n(R)>R^{2n}P^{-\ee}T_3^{-4h},
\eeq
or there exists a positive integer
$s\ll R^2$ such that $\|s\al_3\|<P^{-3}R^2$.  In order to obtain the
result in this form we must remove the weight $\omega(\x/P)$ by
partial summation. We must also verify that the proof of the lemma
still applies when the exponents $\theta$ and $\kappa$ for which 
$R=P^{\theta}$ and $T_3=P^{\kappa/h}$ are not necessarily
constant. Davenport and Lewis require that $0<\theta<1$.  However, if
$R\ge P$ then it is always true that $\|s\al_3\|<P^{-3}R^2$ for some
positive integer $s\le R^2$, by Dirichlet's approximation theorem.
Finally the reader will need to verify that the proof still
goes through for sums of $e(\al_3C(\x)+\al_2Q(\x))$, as opposed to the
terms $e(\al\phi(\x))$ (involving a cubic polynomial $\phi(\x)$)
considered by Davenport and Lewis.

We now present two alternative estimates for $n(R)$. Firstly, for any
form $C$, we can use Lemma 3 of Davenport and Lewis \cite{DL}, which
states that $N(R)\ll R^{2n-h}$.  On the other hand,  if $C$
is non-singular we use Lemma 3 of Heath-Brown \cite{hb-10}, which shows 
that there are $O(R^r)$ integer vectors in the region $|\x|< R$ 
such that the solution set 
\[\{\y\in\RR^n:\,B_i(\x;\b{y})=0\;\forall i\le n\}\]
is $(n-r)$-dimensional. The set will therefore contain $O(R^{n-r})$ integer
vectors with $|\y|<R$, and we deduce that $n(R)\ll R^n$, on summing
for $0\le r\le n$.  Thus $n(R)\ll R^{2n-h}$ in this case too,
since we have defined $h=n$ when $C$ is non-singular.

It now follows that, if we choose $R=P^{\ee}T_3^4$, then
\eqref{eq:nRc} must fail, if $P$ is large enough.  We must therefore
have an integer $s\ll R^2$ for which $\|s\al_3\|<P^{-3}R^2$. We may
therefore write
\beql{eq:t3def}
\al_3=\frac{b_3}{s}+\phi_3
\eeq
with $b_3\in\ZZ$ and $s|\phi_3|<P^{-3}R^2$.  Thus
$s(1+P^3|\phi_3|)\ll R^2$ and on replacing $\ee$
by $\ee/2$ we conclude as follows.
\begin{lemma}\label{lem:W1}
Let $\ee>0$ be given, and define $T_3$ by \eqref{eq:T3def}.  Then
there is a positive integer $s$ such that \eqref{eq:t3def} holds 
with $\gcd(s,b_3)=1$ and 
\[s(1+P^3|\phi_3|)\ll P^{\ee}T_3^8.\]
\end{lemma}

We should emphasise at this point that, as remarked in Section \ref{s:2}, 
we cannot assume that we have $b_3/s=a_3/q$, for the approximation in
\eqref{eq:al2al3}.

We turn now to our second application of Weyl's method.
We shall suppose that \eqref{eq:t3def} holds,
where we think of both $s$ and $\phi_3$ as being small in suitable
senses, and we write
$$
f(\x)=\al_3C(\x)+\al_2Q(\x)
$$
for brevity.  Then
\begin{align*}
S(\al_3,\al_2)&=\sum_{\x\in\ZZ^n}\omega(\x/P)e(f(\x))\\
&=\sum_{\u\bmod{s}}\sum_{\substack{\x\in\ZZ^n\\ \x\equiv\u\bmod{s}}}\omega(\x/P)e(f(\x)).
\end{align*}
 Cauchy's inequality yields
\begin{align*}
|S(\al_3,\al_2)|^2&\le s^n\sum_{\u\bmod{s}}
\left|\sum_{\substack{\x\in\ZZ^n\\ \x\equiv\u\bmod{s}}}\omega(\x/P)e(f(\x))\right|^2\\
&=s^n\sum_{\substack{\x,\y\in\ZZ^n\\ \x\equiv\y\bmod{s}}}\omega(\y/P) \omega(\x/P) 
e(f(\y)-f(\x))\\
&\le s^n\sum_{|\z|<P/s}\left|\sum_{\x\in\ZZ^n}\omega_0(\x/P)
e(f(\x+s\z)-f(\x))\right|,
\end{align*}
where
$\omega_0(\x)=\omega_0(\x,\z)=\omega(\x+sP^{-1}\z)\omega(\x)$. 
Although this remains true even when $s>P$, it is sensible to impose the
condition $s\le P$ for the time being.

Since $C(\x+s\z)-C(\x)$ is automatically
divisible by $s$ we see that
\begin{align*}
e(f(\x&+s\z)-f(\x))\\
&=
e\left(\phi_3\{C(\x+s\z)-C(\x)\}+\al_2\{Q(\x+s\z)-Q(\x)\}\right).
\end{align*}
We now set
\[g(\x)=g(\x,\z)=\phi_3\{C(\x+s\z)-C(\x)\}\]
and conclude that 
\[|S(\al_3,\al_2)|^2\le s^n\sum_{|\z|<P/s}\left|\sum_{\x\in\ZZ^n}\omega_0(\x/P)
e\big(g(\x)+s\al_2\nabla Q(\z).\x\big)\right|.\]
By the Poisson summation formula the inner sum is
\[P^n\sum_{\m\in\ZZ^n}\int_{\RR^n}\omega_0(\t)
e\big(g(P\t)+Ps\al_2\nabla Q(\z).\t-P\m.\t\big)\d\t.\]
The integrals may be estimated by the multidimensional ``first
derivative bound'', see Heath-Brown \cite[Lemma 10]{HB-crelle}, for
example. One has
\[|\nabla\big(g(P\t)+Ps\al_2\nabla Q(\z).\t-P\m.\t\big)|\ge\lambda\]
on ${\rm supp}(\omega_0)$, with
\[\lambda=P|s\al_2\nabla Q(\z)-\m|+O(P^3|\phi_3|).\]
The second and third order derivatives are $O(P^3|\phi_3|)$, and
all higher order derivatives vanish.  It therefore follows from
\cite[Lemma 10]{HB-crelle} that
\[\int_{\RR^n}\omega_0(\t)e\big(g(P\t)+Ps\al_2\nabla Q(\z).\t-P\m.\t\big)\d\t
\ll_A (P|s\al_2\nabla Q(\z)-\m|)^{-A}\]
for any fixed $A>0$, whenever $P|s\al_2\nabla Q(\z)-\m|\gg 
P^3|\phi_3|$. In particular, if $\ee\in (0,1)$ is given, and
\beql{eq:nablabound} 
\|s\al_2\nabla Q(\z)\|\ge P^{-1+\ee}(1+P^3|\phi_3|)
\eeq
then
\begin{align*}
\sum_{\x\in\ZZ^n}\omega_0(\x/P)
&e\big(g(\x)+s\al_2\nabla Q(\z).\x\big) \\
&\ll P^n\sum_{\m\in\ZZ^n}(P|s\al_2\nabla Q(\z)-\m|)^{-A}\\
&\ll 1
\end{align*}
provided that we choose $P$ sufficiently large and take $A>(n+1)/\ee$.
Of course if
\eqref{eq:nablabound} fails then we may estimate the sum trivially as
$O(P^n)$.  We therefore deduce that
\[|S(\al_3,\al_2)|^2\ll s^n\# S_1+s^nP^n\# S_2\] 
with
\[S_1=\{\z\in\ZZ^n:\,|\z|<P/s\}\]
and
\[S_2=\{\z\in\ZZ^n:\,|\z|<P/s,\,\|s\al_2\nabla Q(\z)\|\le 
P^{-1+\ee}(1+P^3|\phi_3|)\}.\]
We may omit the term $\#S_1$ from the above estimate since it is at most
$O(P^n)$, while $S_2$ contains at least the element $\z=\mathbf{0}$.
If $|\z|<P/s$ one has $|\nabla Q(\z)|\le cP/s$, for some constant
$c=c(Q)$.  
We now recall the notation 
\[\rho=\rank(Q)\]
introduced earlier.
Thus the values $\nabla Q(\z)$ are restricted to a
vector space of dimension $\rho$.  Given $\w$, the equation $\w=\nabla 
Q(\z)$ has $O((P/s)^{n-\rho})$ integral solutions $\z$ with $|\z|<P/s$, 
and we conclude that 
$$
\#S_2\ll  (P/s)^{n-\rho}\cN^{\rho},
$$
where
\[\cN=\#\{w\in\ZZ:\,|w|\le cP/s,\,\|s\al_2w\|\le 
P^{-1+\ee}(1+P^3|\phi_3|)\}.\]
We therefore have
\[|S(\al_3,\al_2)|^2\ll P^{2n-\rho}s^{\rho}\cN^{\rho}.\]
We now define $T_2=T_2(\al_3,\al_2)$ by setting
\beql{eq:T2def}
|S(\al_3,\al_2)|=P^nT_2^{-\rho},
\eeq
whence
\beql{eq:T2B}
T_2^2\gg P/(s\cN).
\eeq

Naturally our next task is to estimate $\cN$. Generally, if
\[\mathcal{W}=\{w\in\ZZ:\,|w|\le W,\,\|\mu w\|\le \xi\},\]
then $\# \mathcal{W}$ is at most the number of points of the lattice
\[\Lambda=\{\big(W^{-1}u,\xi^{-1}(\mu u-v)\big):\, (u,v)\in\ZZ^2\}\]
lying in the unit square. The determinant of the lattice is
$(W\xi)^{-1}$, and so the number of points is $O(1+W\xi+\sigma^{-1})$,
where $\sigma$ is the first successive minimum of the lattice. From the
definition of $\sigma$ we see that there will be a non-zero point
$(u,v)\in\ZZ^2$ such that $|u|\le\sigma W$ and $|\mu u-v|\le\sigma\xi$.

Thus in our situation we find that
\[\cN\ll 1+P^{\ee}s^{-1}(1+P^3|\phi_3|)+\sigma^{-1}\]
so that either $\cN\ll 1+P^{\ee}s^{-1}(1+P^3|\phi_3|)$ or
$\sigma\le\cN^{-1}$.  In the former case \eqref{eq:T2B} yields
\[T_2^2\gg
\min\left(\frac{P}{s}\,,\,\frac{P^{1-\ee}}{1+P^3|\phi_3|}\right)
\gg\frac{P^{1-\ee}}{s+P^3|\phi_3|}.\]
In the latter case \eqref{eq:T2B} shows that there is a non-zero point $(u,v)$
with 
\[|u|\le cP/(s\cN)\ll T_2^2\] 
and 
\[|s\al_2 u-v|\le P^{-1+\ee}(1+P^3|\phi_3|)/\cN\ll
P^{-2+\ee}s(1+P^3|\phi_3|)T_2^2.\]
If there is any such point $(u,v)$ for which $u=0$, then $v\not=0$
whence we must have $P^{-2+\ee}s(1+P^3|\phi_3|)T_2^2\gg 1$.  But in
that case we may take $u=1$ and we will automatically have $\|s\al_2u\|
\ll P^{-2+\ee}s(1+P^3|\phi_3|)T_2^2$.  Thus we can assume with no
loss of generality that there is a solution in which $u\not=0$.
We now summarise our findings as follows.
\begin{lemma}\label{lem:Weyl3}
Define
\[|S(\al_3,\al_2)|=P^nT_2^{-\rho}\]
and suppose that \eqref{eq:t3def} holds with $\gcd(s,b_3)=1$.
Then for any fixed $\ee>0$ one of the following must happen:
\begin{itemize}
\item[(i)] there is a positive integer $u\ll T_2^2$ such that 
\[\|su\al_2\|\ll P^{-2+\ee}s(1+P^3|\phi_3|)T_2^2; \]
or
\item[(ii)] we have
\[T_2^2\gg \frac{P^{1-\ee}}{s+P^3|\phi_3|}.\]
\end{itemize}
\end{lemma}
Note that we assumed that $s\le P$ during the proof.  However the result
is clearly trivial when $s\ge P$ since we then have $T_2^2\gg 1\gg P/s$.

\section{Minor arc contribution: the Weyl bound}\label{s:4a}

In this section we will see what can be said about the size of the minor
arc integral \eqref{eq:minor} on the basis of Lemmas \ref{lem:W1} and
\ref{lem:Weyl3}. For convenience we write
\[I(\minor):=\iint_{\minor} S(\al_3,\al_2) \d\al_3\d \al_2 .\]
We begin by considering values $\al_3$ for which case (i) of
Lemma \ref{lem:Weyl3} holds.  

Our first move is to show that on the minor arcs $T_3$ (and 
hence also $T_2$) cannot be too small. Lemma \ref{lem:W1} and case (i)
of Lemma \ref{lem:Weyl3} produce positive integers $s$ and $u$ 
such that
\[su\ll P^{\ee}T_3^8T_2^2.\] 
Moreover there will be integers $b_3,b_2$ for which
\[|su\al_3-ub_3|=su|\phi_3|\ll 
uP^{-3+\ee}T_3^8\ll P^{-3+\ee}T_3^8T_2^2\]
and
\[|su\al_2-b_2|=\|su\al_2\|\ll P^{-2+\ee}s(1+P^3|\phi_3|)T_2^2\ll
P^{-2+2\ee}T_3^8T_2^2.\]
It follows that we would have $su\le P^{\delta}$ and
\begin{align*}
\left|\al_3-\frac{b_3u}{su}\right| 
&\le 
|su\al_3-ub_3|\le P^{-3+\delta},\\
\left|\al_2-\frac{b_2}{su}\right| 
&\le|su\al_2-b_2|\le  
P^{-2+\delta},
\end{align*} 
if $T_3^8T_2^2\le P^{\delta-3\ee}$ say, with $P$ sufficiently
large. Thus if $(\al_3,\al_2)\in\minor$ we must have $T_3^8T_2^2\ge
P^{\delta-3\ee}$. It is clear from \eqref{eq:T3def} and
\eqref{eq:T2def} that
\beql{eq:T2T3}
T_2=T_3^{h/\rho}.
\eeq
We therefore deduce that
\beql{eq:T3lb}
T_3\ge P^{\delta\rho/(16\rho+4h)}
\eeq
provided that $\ee\le\delta/6$, as we henceforth assume.

In estimating the minor arc integral $I(\minor)$, it will be convenient to 
 consider the contribution $I_{t_3}(\minor)$, say, from all pairs
 $\al_3,\al_2$ for which $T_3$ lies in a dyadic range 
$$
t_3<T_3\le 2t_3.
$$ 
In view of \eqref{eq:T3lb} we may assume that
$t_3\ge P^{\delta\rho/(16\rho+4h)}$. Moreover, \eqref{eq:T2T3} implies that 
$$
t_3^{h/\rho}<T_2\leq (2t_3)^{h/\rho}.
$$
We put  $t_2=t_3^{h/\rho}.$

We proceed to consider the contribution to $I_{t_3}(\minor)$ 
from all  pairs $\al_3,\al_2$  for which the first alternative of Lemma
\ref{lem:Weyl3} holds. 
We begin by considering the
measure of the available $\al_2\in(0,1]$. 
For each positive integer
$u\ll t_2^2$ there will be an integer $v\ll su$ such that
\[|su\al_2-v|\ll P^{-2+\ee}s(1+P^3|\phi_3|)t_2^2.\]
Thus the total measure for the values of $\al_2$ will be
\[
\ll\sum_{u\ll t_2^2}\sum_{v\ll su}(su)^{-1}P^{-2+\ee}s(1+P^3|\phi_3|)t_2^2
\ll P^{-2+\ee}s(1+P^3|\phi_3|)t_2^4.\]
According to Lemma \ref{lem:W1} we will have $s(1+P^3|\phi_3|)\ll
P^{\ee}t_3^8$ so that the above is $O(P^{-2+2\ee}t_3^8t_2^4)$. We may
calculate the available measure for $\al_3$ in much the same way,
given that $s\ll P^{\ee}t_3^8$ and $|\phi_3|\ll
P^{-3+\ee}s^{-1}t_3^8$, by Lemma \ref{lem:W1}.  This yields
\begin{align}
\meas\{\al_3:\, s(1+P^3|\phi_3|)\ll P^{\ee}t_3^8\}
&\ll  \sum_{s\ll P^{\ee}t_3^8}\;\sum_{v\ll s}P^{-3+\ee}s^{-1}t_3^8\nonumber\\
&\ll P^{\ee}t_3^8.P^{-3+\ee}t_3^8\nonumber\\
&= P^{-3+2\ee}t_3^{16}.
\label{eq:al3meas}
\end{align}

Returning to our estimation of the  contribution to $I_{t_3}(\minor)$
from the first case of Lemma
\ref{lem:Weyl3}, 
we obtain the overall contribution 
\[
\ll P^nt_3^{-h}.P^{-2+2\ee}t_3^8t_2^4.P^{-3+2\ee}t_3^{16}\ll
P^{n-5+4\ee}t_3^{-h+4h/\rho+24}.\]
If $(h-24)(\rho-4) >96$ then $h-4h/\rho-24\ge 1/\rho$.  Thus
if we sum over all relevant dyadic ranges for $t_3\ge P^{\delta\rho/(16\rho+4h)}$,
we will obtain an overall contribution 
\[\ll P^{n-5+4\ee-\delta/(16\rho+4h)},\]
which is satisfactory if we choose $\ee$ sufficiently small.
We record our conclusions as follows.

\begin{lemma}\label{lem:Weyl4}
The contribution to $I(\minor)$ arising
from pairs $\al_3,\al_2$ for which the first alternative of Lemma
\ref{lem:Weyl3} holds, is $o(P^{n-5})$ providing that
\[(h-24)(\rho-4) >96.\]
\end{lemma}

Turning to the contribution to $I_{t_3}(\minor)$ from 
those pairs $\al_3,\al_2$  for which the second alternative of
Lemma~\ref{lem:Weyl3} holds,  we first consider
the situation when $t_3\ge P^{3/19}$. 
 According to \eqref{eq:al3meas} the available set of values for
$\al_3$ has measure $O(P^{-3+2\ee}t_3^{16})$, but there is no
restriction on the values of $\al_2$.  It follows that the
contribution to the minor arc integral is
\beql{eq:c2b}
\ll P^nt_3^{-h}.P^{-3+2\ee}t_3^{16}.
\eeq
We proceed to sum over dyadic values $t_3\ge P^{3/19}$ to obtain a total
\[\ll P^{n-3+2\ee-3(h-16)/19}\le P^{n-5+2\ee-1/19},\]
provided that $h\ge 29$.  This gives us the following result.
\begin{lemma}\label{lem:Weyl5}
Suppose that $h\ge 29$.  Then the contribution to $I(\minor)$
arising from pairs $\al_3,\al_2$ for which
the second case of Lemma \ref{lem:Weyl3} holds, and $T_3\ge P^{3/19}$,
is $o(P^{n-5})$.
\end{lemma}

The simplest way to handle the remaining case is to combine the
inequalities $s(1+P^3|\phi_3|)\ll P^{\ee}T_3^8$ and $T_2^2\gg
P^{1-\ee}/(s+P^3|\phi_3|)$ from Lemma \ref{lem:W1} and part (ii) of Lemma
\ref{lem:Weyl3}, respectively, to deduce that
\[T_3^8T_2^2\gg P^{1-2\ee}\frac{s+sP^3|\phi_3|}{s+P^3|\phi_3|}\ge
P^{1-2\ee}.\]
Then \eqref{eq:T2T3} implies that $T_3^{8+2h/\rho}\gg P^{1-2\ee}$.  We
therefore see from the bound \eqref{eq:c2b} that the total contribution
to the minor arc integral is $O(P^{n-\psi})$ with
\[\psi=3-2\ee+(h-16)\frac{1-2\ee}{8+2h/\rho}.\]
By taking $\ee$ sufficiently small we can make $\psi>5$ provided that
\[h-16>2(8+2h/\rho).  \]
This gives us the following lemma,
which is exactly what we need for Theorem \ref{main'}.

\begin{lemma}\label{lem:Weyl6}
Suppose that $(h-32)(\rho-4)>128$.  Then  
\[I(\minor)=o(P^{n-5}).\]
\end{lemma}

An alternative way to deal with the case $T_3\le P^{3/19}$ is to use
an analysis based on the Poisson summation formula.  We will carry
this out in Section \ref{s:7}.  
It is an essential feature of the method that one uses
simultaneous rational approximations $a_3/q, a_2/q$ to $\al_3$ and
$\al_2$, as given by \eqref{eq:al2al3}.

We will want to know whether the approximation $a_3/q$ corresponds to the
approximation $b_3/s$ given by \eqref{eq:t3def}.  However if
$b_3/s\not=a_3/q$ then
\[\frac{1}{sq}\le\left|\frac{a_3}{q}-\frac{b_3}{s}\right|\le
\left|\al_3-\frac{a_3}{q}\right|+\left|\al_3-\frac{b_3}{s}\right|\le
\frac{1}{qQ_3}+|\phi_3|.\]
It would then follow from Lemma \ref{lem:W1} and \eqref{eq:Q2Q3} that
\begin{align*}
1&\le s/Q_3+sq|\phi_3|\\
&\ll P^{\ee}T_3^8(Q_3^{-1}+P^{-3}Q_3Q_2)\\
&\ll P^{24/19+\ee}.P^{-4/3}\\
&\le P^{-4/57+\ee},
\end{align*}
providing that $T_3\le P^{3/19}$.
This will produce a contradiction if $\ee$ is small enough and $P$ is
large enough, thereby proving that $a_3/q=b_3/s$.

We record this conclusion as follows.
\begin{lemma}\label{lem:sq}
Suppose that $h\ge 29$ and $T_3\le P^{3/19}$. Then we will have 
$a_3/q=b_3/s$ if $P$ is large enough.
\end{lemma}

\section{Poisson summation}\label{s:7}

In this section we suppose that $X\subset \PP^{n-1}$ is
non-singular. By Corollary \ref{c:3.2} 
we may assume that the cubic form $C$  is 
non-singular and that $Q$ 
takes the shape
\beql{eq:didef}
Q(\x)=\sum_{i=1}^nd_ix_i^2,
\eeq
with $d_1,\ldots,d_n\in \ZZ$ such that 
$d_1\cdots d_{n-1}\not=0$.
Thus $Q$ has rank at least $n-1$. 
We are now ready to begin our  analysis of the exponential sums 
$$
S(\al_3,\al_2)=\sum_{\x\in\ZZ^n}\omega(\x/P) e\left(\al_3 C(\x)+\al_2 Q(\x)\right),
$$
for $\alpha_3,\alpha_2\in \RR$, 
based on an application of Poisson summation. 

We will assume throughout this section that
$\al_3=a_3/q+\theta_3$ and 
$\al_2=a_2/q+\theta_2$, as 
 in Section \ref{s:2}.  Thus $\a=(a_3,a_2)\in \ZZ^2$ and $q\in \ZZ$   satisfy
\begin{equation}
  \label{eq:aiq}
1\leq a_3,a_2\leq q\leq Q_3Q_2, \quad \gcd(q,\a)=1,
\end{equation}
and $\bth=(\theta_3,\theta_2)\in\RR^2$ satisfies 
\begin{equation}\label{eq:theta}
  |\theta_i|\leq q^{-1}Q_i^{-1}, \quad (i=3,2).
\end{equation}
We recall that $Q_3,Q_2$ are positive integers
given by \eqref{eq:Q2Q3}. 
Our first step involves introducing complete
exponential sums modulo $q$.  The following 
result is standard.

\begin{lemma}\label{lem:p_sum}
We have
$$
  S(\alpha_3,\alpha_2)
  =\frac{P^n}{q^{n}}\sum_{\m\in\ZZ^n}S(\a,q;\m)I(\theta_3 P^3,\theta_2
  P^2;q^{-1}P\m), 
$$
where
\begin{align}
  \label{eq:Sq}
S(\a,q;\m)&:=\sum_{\y\bmod{q}}e_q(a_3C(\y)+a_2Q(\y)+\m.\y),\\
  \label{eq:Iq}
  I(\bga;\z)&:=\int_{\RR^n} \omega(\x)e(\gamma_3 C(\x)+\gamma_2 Q(\x)-\z.\x)\d\x.
\end{align}
\end{lemma}

\begin{proof}
Write $\x=\y+q\z$, for $\y \bmod{q}$.
Then we obtain
\begin{align*}
  S(\alpha_3,\alpha_2)=~&
  \sum_{\y\bmod{q}}e_q(a_3C(\y)+a_2Q(\y))\\
  &\times
\sum_{\z\in\ZZ^n}\omega((\y+q\z)/P)e(\theta_3C(\y+q\z)+
\theta_2Q(\y+q\z)).
\end{align*}
The statement of the lemma follows from an application of Poisson
summation, followed by an obvious  change of variables.  
\end{proof}

\bigskip

We begin by analysing 
the complete exponential sums $S(\a,q;\m)$ given by \eqref{eq:Sq}, for
$\gcd(q,\a)=1$ and $\m\in \ZZ^n$.  
They satisfy the multiplicativity property 
\begin{equation}\label{eq:mult}
S(\a,rs;\m)=S(\a_s,r;\m)S(\a_r,s;\m), \quad \mbox{for $\gcd(r,s)=1$},
\end{equation}
where 
\[\a_t:=(t^2a_3,ta_2).\]
The proof of this fact is standard (see \cite[Lemma 10]{41}, for example).
In view of this it will suffice to analyse 
$S(\a,q;\m)$ for prime power values of $q$.

It will be convenient to give a separate treatment of the moduli $q$
that are built from prime divisors of $a_3$. 
Recall the shape \eqref{eq:didef} that $Q$ takes, with 
$d_1\cdots d_{n-1}\not=0$. For a given prime $p$ we let $p^v$ be the
largest power of $p$ dividing any of 
$2d_1,\ldots,2d_{n-1}$.  Since $Q$ is fixed we will have $p^v\ll 1$.  We can
now state our result.

\begin{lemma}\label{lem:bad-prime}
Suppose that $p^{1+v}\mid a_3$ and let $r\geq 1$.
Then for  any  $\m\in 
\ZZ^n$, we have 
$$
S(\a,p^r;\m)\ll p^{r(n+1)/2}.
$$
\end{lemma}

\begin{proof}
Since $p\mid a_3$ we may assume that $p\nmid a_2$.  Let $S=S(x)$ be the sum
\begin{align*}
\sum_{x_1,\ldots,x_{n-1}\bmod{p^r}}
\hspace{-0.3cm}
e_{p^r}\big(a_2Q(x_1,\ldots,x_{n-1},x)+a_3C(x_1,\ldots,x_{n-1},x)\big).
\end{align*}
We will show that $S\ll p^{r(n-1)/2}$ for every $x$, which will suffice. 
Our approach is based on applying Weyl's method to  
$|S|^2$.  This gives
\beql{eq:s2nn}
|S|^2\leq \sum_{y_1,\ldots,y_{n-1}
  \bmod{p^r}}\left|\sum_{x_1,\ldots,x_{n-1}\bmod{p^r}} 
e_{p^r}(f)\right|
\eeq
where $f=f(x_1,\ldots,x_{n-1};y_1,\ldots,y_{n-1})$ has the shape
\[2a_2\sum_{i=1}^{n-1}d_ix_iy_i+
p^{1+v}\sum_{i=1}^{n-1}y_ig_i(x_1,\ldots,x_{n-1};y_1,\ldots,y_{n-1}),\]
since $p^{1+v}\mid a_3$. Here the $g_i$ are
suitable polynomials defined over $\ZZ$.

Suppose now that we have an exponent $h\le r-v-1$ such that
$p^h|y_1,\ldots,y_{n-1}$, but some $y_i$ is not divisible by
$p^{h+1}$.  Let us suppose that $p^{h+1}\nmid y_1$, say.  Writing
$x_1=s+p^{r-h-v-1}t$, with $s$ running modulo $p^{r-h-v-1}$ and $t$
modulo $p^{h+v+1}$, one finds that
$$
f\equiv 2a_2d_1y_1p^{r-h-v-1}t+
f_0(s;x_2,\ldots,x_{n-1};y_1,\ldots,y_{n-1}) \bmod{p^r},
$$
for some integral polynomial $f_0$.  It follows that the sum over $t$
vanishes unless $p^{h+v+1}\mid 2a_2d_1y_1$. However this latter
condition would contradict
the facts that $p\nmid a_2$, $p^{v+1}\nmid 2d_1$ and $p^{h+1}\nmid y_1$.

We therefore deduce that the inner sum of \eqref{eq:s2nn} vanishes
unless $p^{r-v}$ divides each of $y_1,\ldots,y_{n-1}$.  There are
therefore $p^v\ll 1$ choices for each of these, and for each such
choice the inner sum has modulus at most $p^{r(n-1)}$.  We then deduce
that $|S|^2\ll p^{r(n-1)}$, and the lemma follows.
\end{proof}

We are now ready to begin in earnest our  treatment of the exponential
sum $S(\a,q;\m)$ for $q\in \NN$.  
Let us write $q=q_0q_1q_2$, where
\beql{eq:q02} 
q_0=\prod_{\substack{p^e\| q\\ p^{1+v}\mid a_3}} p^e,\;\;\;
q_2=\prod_{\substack{p^e\| q,\,e\ge 3\\ p^{1+v}\nmid a_3}} p^e.
\eeq
Thus $q_1$ is cube-free, and $\gcd(q_1q_2,a_3)$
divides $\prod_p p^v$, which in turn divides $2\prod_{i=1}^{n-1}d_i,$
where $d_i$ are the coefficients of $Q$.  It follows that
$\gcd(q_1q_2,a_3)\ll 1$.

Lemma \ref{lem:bad-prime} and \eqref{eq:mult} will suffice to deal with
the sum associated to the modulus $q_0$.  
The cube-free modulus $q_1$ will be handled via the following result. 

\begin{lemma}\label{lem:sq-free}
Let $\ve>0$. 
Suppose that $q$ is cube-free, and is a product of primes $p$ for
which $p^{1+v}\nmid a_3$.
Then for any  $\m\in 
\ZZ^n$, we have 
$$
S(\a,q;\m)\ll q^{n/2+\ve}.
$$
\end{lemma}

\begin{proof}
By \eqref{eq:mult} it will suffice to show that 
$
S(\a,p^r;\m)\ll p^{rn/2},
$
for $r\in \{1,2\}$ and 
each prime with $p^{1+v}\nmid a_3$. The result is trivial for the
finitely many primes with $v\not=0$.  Indeed we may assume that 
$p\gg 1$, where the
implied constant is taken  
 large enough to ensure that 
$C$ is non-singular modulo $p$.
When $r=2$ the result therefore follows from work of Heath-Brown
\cite{HB-canada}. 
Suppose next that $r=1.$ 
We wish to apply the estimate
\[\sum_{\x\in \FF_p^n}e_p(f(\x))\ll_{d,n} p^{n/2},\]
of Deligne \cite{deligne}, which applies to any polynomial $f$ over
$\FF_p$ of degree $d$, in $n$ variables, whose leading homogeneous part is
non-singular modulo $p$.  
Taking $f(\x)= a_3C(\x)+a_2Q(\x)+\m.\x$ we get 
$S(\a,p;\m)\ll p^{n/2}$, as required.
\end{proof}

It is now time to turn our attention to the cube-full modulus $q_2$, with
$\gcd(q_2,a_3)\ll 1$.  
Our next goal is  the  following  variant of 
 \cite[Lemma~14]{hb-10}.

\begin{lemma}\label{lem:anaL14}
Let $\ve>0$ and let $\m_0\in 
\RR^n$. 
Suppose that $q$ is cube-full, with $\gcd(q,a_3)\ll 1$.
Then we have 
\[
\sum_{|\m-\m_0|\le V}|S(\a,q;\m)|\ll  
q^{n/2+\ve}\left\{
V^n+q^{n/3} \right\},
\] 
for any $V\geq 1$.
\end{lemma}

For the proof we will
modify parts of the argument from Browning and Heath-Brown \cite[\S
  5]{41}.  We will be fairly brief, since the changes necessary are
minor, if somewhat tedious. We will write our square-full modulus 
$q$ as $q=c^2d$ with $d$ square-free. 

Firstly, in analogy to \cite[Lemma 11]{41}, one may show that
\begin{equation}\label{eq:rhodius}
|S(\a,q;\m)|\leq 
(c^2d)^{n/2}
\sum_{\substack{\ma{u}\bmod{c}\\ c \mid (\nabla g(\ma{u})+\m)}}M_{d}(\ma{u})^{1/2},
\end{equation}
where 
\[g(\u)=a_3C(\u)+a_2Q(\u)\]
and
$$
M_d(\u)=\#\left\{\x \bmod{d}: \nabla^2 g(\u).\x 
\equiv \ma{0} \bmod{d}\right\}.
$$
Corresponding to the sum $\mathcal{S}(V,a;\m_0,c,d)$ in
\cite[Eq.~(5.9)]{41} we define
\[\mathcal{S}(V)=\mathcal{S}(V,a_3,a_2;\m_0,c,d):=
\sum_{|\m-\m_0|\le  V} 
\sum_{\substack{\b{a}\bmod{c}\\c \mid (\nabla  g(\ma{a})+\m)}}M_{d}(\ma{a})^{1/2}.\]

We would like to adapt \cite[Lemma 16]{41} to our present situation. 
Note that \cite[Section 5]{41} is concerned with exponential sums 
associated to general cubic polynomials $g$ for which the cubic part
is non-singular and 
$\|g\|_P=\|P^{-3}f(Px_1,\ldots,Px_n)\|\leq H$ for some parameter $H$. 
In our setting one may verify that it is possible to replace $H$ by 1
in the various estimates of \cite{41}. A number of trivial adjustments
need to be made since we have $\gcd(c^2d,a_3)\ll 1$, rather than
$\gcd(c^2d,a)=1$. 

Note that \cite[Lemma 13]{41} can be applied directly with $H\ll 1$
since it pertains only to the cubic part $C$ of $g$, where
$\|g_0\|_P=\|g_0\|\ll 1$. 
Moreover, \cite[Lemma 14]{41} can also be applied, with 
$D\ll 1$. 
Turning to the analogue of \cite[Lemma 16]{41}, 
which relies on 
\cite[Lemmas 13 and 14]{41}, the proof goes through unchanged, 
with $D\ll 1$ and $H\ll
1$. 
For any $\ve>0$, this leads to the estimate 
\begin{align*}
\mathcal{S}\ll q^\ve V^{n}\Big(1+ \frac{c^2d}{V^3}\Big)^{n/2}.
\end{align*}
But then, taking
\[V_1=V+(c^2d)^{1/3},\]
we have
\[\mathcal{S}(V)\le\mathcal{S}(V_1)
\ll q^\ve V_1^{n}\Big(1+ \frac{c^2d}{V_1^3}\Big)^{n/2}
\ll q^\ve \big(V^{n}+ (c^2d)^{n/3}\big).\]
Lemma \ref{lem:anaL14} now follows from \eqref{eq:rhodius}.

\bigskip

We next turn to the analysis of the exponential integral 
\begin{align*}
I&=I(\theta_3P^3,\theta_2P^2;q^{-1}P\m)\\
  &= \int_{\RR^n} \omega(\x)e\big(\theta_3 P^3C(\x)+\theta_2
P^2Q(\x)-q^{-1}P\m.\x\big)\d\x. 
  \end{align*}
For this it will be convenient to write
$$
f(\x)=\theta_3 P^3C(\x)+\theta_2 P^2Q(\x).
$$ 
We will proceed by 
adapting the proof of  \cite[Lemma 6]{41}, 
noting that our weight function $\omega$ belongs to the class of
weight functions considered therein.  
Let $\nu\in \RR$ be a parameter in the range $0<\nu\leq 1$, to be
chosen in due course. 
We decompose 
$I$ into an average of integrals over subregions
of size at most $\nu$. It follows from \cite[Lemma 2]{HB-crelle}
that there exists an infinitely differentiable  weight function
$w_\nu(\x,\y):\RR^{2n}\rightarrow \RR_{\geq 0}$, 
such that
$$
\omega(\x)=\nu^{-n}\int_{\RR^n} w_{\nu}\Big(\frac{\x-\y}{\nu},\y\Big)\d\y.
$$
Moreover, $\supp(w_\nu)\subseteq [-1,1]^n\times \supp(\omega)$.
Then on making this substitution into $I$, and writing $\x=\y+\nu \u$, 
we obtain
\begin{equation}\begin{split}
\label{eq:lemIq:1}
|I|
&=\nu^{-n}\Big|\int_{\RR^n}\int_{\RR^n}
w_\nu\big(\nu^{-1}(\x-\y),\y\big)e(f(\x)-q^{-1}P\ma{m}.\x)\d\x\d\y\Big|
\\
&\leq\int_{\RR^n}\Big|\int_{\RR^n} w_\nu(\u,\y)e(f(\y+\nu \u)-\nu q^{-1}P
\ma{m}.\u)\d\u\Big|\d\y\\
&=\int_{\supp(\omega)} |K(\y)|\d\y,
\end{split}\end{equation}
say. 

Let us 
write $F(\u)=f(\y+\nu \u)-\nu q^{-1}P\ma{m}.\u$, for fixed $\y$.
It is clear that $f(\y+\nu\u)\ll \Theta $  for any 
$(\y,\u)\in \supp(\omega)\times [-1,1]^n$, 
where we put 
\begin{equation}\label{eq:def-T}
\Theta = 1+|\theta_3|P^3+|\theta_2|P^2.
\end{equation}
For such $(\y,\u)$  
it follows that the $k$-th order
derivatives of $F(\u)$ are all $O_k(\nu^k\Theta  )$, for $k\geq 2$.
Likewise, one finds that
$$
\nabla F(\u) = \nu \nabla f(\y) - \nu q^{-1}P \ma{m} +O(\nu^2 \Theta ).
$$
Let $R\geq 1$ and suppose that $|
 \nabla f(\y) - q^{-1}P \ma{m}|\geq \nu^{-1}R$.  Then it follows that 
there exists a constant $c(n)>0$
such that $|\nabla F(\u)|\gg R$, provided that 
$$
R\geq c(n) \nu^2 \Theta .
$$
We will take $\nu=\Theta ^{-1/2}$, so that $0<\nu\leq
1$. An application of \cite[Lemma 10]{HB-crelle} now reveals that
$K(\y)\ll_{N} R^{-N}$ for any $N\geq 1$, when $R\geq c(n)$.
Inserting this into \eqref{eq:lemIq:1} gives
$$
I \ll_{N} R^{-N}+ \meas \mathcal{S}(R) 
$$
for any $N\geq 1$ and  any $R\geq c(n)$, where we have written
\[\mathcal{S}(R)= 
\left\{\y\in\supp(\omega):~ |\nabla f(\y) - q^{-1}P \ma{m}|
\leq R \sqrt{\Theta }\right\}.\]
If we choose $R=P^{\ve}$ with
some fixed $\ve>0$ then $R^{-N}$ can be made smaller than any given
negative power of $P$, via a  suitable choice of $N$.  This leads to the
following result.
\begin{lemma}\label{lem-expint}
Let $\ve>0$ and $N\in\NN$ be given.  Then
\[I(\theta_3P^3,\theta_2P^2;q^{-1}P\m)\ll_{N}
P^{-N}+ \meas \mathcal{S}(P^{\ve}). \] 
\end{lemma}

Alternatively, if $|\ma{m}|\ge c\Theta q/P$ with a suitably large constant
$c$, then $|\nabla f(\y)|\le\tfrac12 q^{-1}P|\ma{m}|$ for
$\y\in\supp(\omega)$.  Thus, if we take 
\[R=\tfrac13 q^{-1}P|\ma{m}|\Theta ^{-1/2},\]
say, then $\mathcal{S}(R)$ 
will be empty.  Moreover, if $|\ma{m}|\ge P^{\ve-1}q\Theta $
for some positive $\ve<1$ then we will have
\[\frac{R}{(P|\ma{m}|)^{\ve/3}}=\frac{(P|\ma{m}|)^{1-\ve/3}}{3q\sqrt{\Theta }}
\ge\frac{P^{\ve(1-\ve/3)}\Theta ^{1/2-\ve/3}}{3q^{\ve/3}}
\ge\frac{P^{\ve(1-\ve/3)}}{3P^{2\ve/3}}\ge \frac{1}{3},\]
since $\Theta \ge 1$ and $q\le P^2$.  It follows that $R\gg
(P|\ma{m}|)^{\ve/3}$ whenever $|\ma{m}|\ge P^{\ve-1}q\Theta $.
This leads to the following conclusion.
\begin{lemma}\label{lem:int-tail}
Let $\ve>0$ and let $N\in\NN$ be given.  Then
\[I(\theta_3P^3,\theta_2P^2;q^{-1}P\m)\ll_{N}
P^{-N}|\ma{m}|^{-N}\]
whenever $|\ma{m}|\ge P^{\ve-1}q\Theta $.
\end{lemma}

\bigskip
We are now ready to deduce a final estimate for the exponential sum
$S(\alpha_3,\alpha_2)$ in \eqref{eq:S}, for any $(\alpha_3,\alpha_2)\in \RR^2$.
We suppose as before that
$\alpha_i=a_i/q+\theta_i$, with $\a=(a_3,a_2)\in \ZZ^2$ and $q\in \ZZ$
satisfying \eqref{eq:aiq} and 
 $\bth=(\theta_3,\theta_2)\in\RR^2$ satisfying
\eqref{eq:theta}. Here $Q_3,Q_2\in \NN$ are given by \eqref{eq:Q2Q3}.

Our starting point is Lemma \ref{lem:p_sum}.  We use Lemma
\ref{lem:int-tail} to handle the tail of the summation over $\ma{m}$,
so that
\[S(\alpha_3,\alpha_2)\ll 1+
\frac{P^n}{q^{n}}\sum_{|\m|\le P^{\ve-1}q\Theta }|S(\a,q;\m)|.
|I(\theta_3 P^3,\theta_2 P^2;q^{-1}P\m)|\]
for any fixed $\ve>0$.
Next we employ the multiplicativity property \eqref{eq:mult} in
conjunction with Lemmas \ref{lem:bad-prime} and \ref{lem:sq-free} to
show that the second term is 
$$
\ll \frac{P^{n+\ve}}{q^{n}}q_0^{(n+1)/2}q_1^{n/2}
\sum_{|\m|\le P^{\ve-1}q\Theta }
\hspace{-0.3cm}
|S(\a_{q_0q_1},q_2;\m)|. 
|I(\theta_3 P^3,\theta_2 P^2;q^{-1}P\m)|.
$$
We then use Lemma \ref{lem-expint} which shows that this is
$$
\ll 1+\frac{P^{n+\ve}}{q^{n}}q_0^{(n+1)/2}q_1^{n/2}
\int_{\supp(\omega)}
\left(\sum_{|\m-\m_0(\y)|\le V}|S(\a_{q_0q_1},q_2;\m)|\right)\d\y,
$$
where
\[\m_0(\y)=q\big(\theta_3 P^2\nabla C(\y)+\theta_2 P\nabla Q(\y)\big)\]
and
\[V=P^{\ve-1}q\sqrt{\Theta }.\]
Finally Lemma \ref{lem:anaL14} produces the bound
\[S(\alpha_3,\alpha_2)\ll 1+
\frac{P^{n+\ve}}{q^{n}}q_0^{(n+1)/2}q_1^{n/2}q_2^{n/2+\ve} 
\left\{V^n+q_2^{n/3} \right\}.\]
We have therefore established the following result, on re-defining $\ve$.

\begin{lemma}\label{lem:main-poisson}
Let $\ve>0$, and let $q_0,q_2$ and $\Theta $ be defined as in \eqref{eq:q02} and
\eqref{eq:def-T}, respectively.
Then we have  
\[S(\alpha_3,\alpha_2)\ll 
q_0^{1/2}P^{\ve}\{q^{n/2}\Theta ^{n/2}+P^nq^{-n/2}q_2^{n/3}\}.\]
\end{lemma}

We should comment at this point that the first term on the right is
more or less what one would hope for. The second term on the right
could probably be improved, but suffices for our purposes. When $q$
and $\Theta$ are both of order 1, 
so that $(\al_3,\al_2)$ lies in the major arcs, 
we expect that $S(\alpha_3,\alpha_2)$ is approximately
$P^nq^{-n}S(\a,q,\b{0})I(0,0;\b{0})$.  This corresponds to the second
term on the right in Lemma \ref{lem:main-poisson}.  However when $\Theta $
is a little larger than 1 we would expect to have a non-trivial bound for
$I(\theta_3 P^3,\theta_2P^2;\b{0})$, and the lemma does not take any
account of this.

\section{Minor arc contribution: 
Theorem \ref{main}}

In this section we will combine the work of Sections \ref{s:4a} and \ref{s:7}
so as to handle the minor arcs for Theorem \ref{main}.  Thus we will
assume that $h=n\geq 29$ and that $\rho\ge n-1$. Lemma \ref{lem:Weyl4}
then gives 
a satisfactory result under the first alternative of Lemma
\ref{lem:Weyl3}, provided that $n\ge 29$.  Indeed we see that one
cannot hope to handle the case $n=28$, even when $\rho=n$. 
Moreover Lemma \ref{lem:Weyl5} gives a satisfactory result 
under the second alternative of Lemma 
\ref{lem:Weyl3} when $T_3\geq P^{3/19}$.
We therefore investigate the second alternative of Lemma
\ref{lem:Weyl3}, under the assumption that $T_3\leq P^{3/19}$ and $n\geq 29$.
Furthermore, it follows from Lemma 
\ref{lem:sq} that we may proceed under the assumption that 
$b_3/s=a_3/q$.

We will assume that  $T_3$ lies in a dyadic range
$
t_3<T_3\le 2t_3,
$ 
with 
\begin{equation}\label{eq:flight}
P^{\delta \rho/(20n)}\leq P^{\delta \rho/(16\rho+4h)}\leq t_3\leq P^{3/19}, 
\end{equation}
where the lower bound comes from 
\eqref{eq:T3lb}.
It follows from \eqref{eq:T2T3} that 
$
t_2<T_2\leq 2^{h/\rho}t_2,
$
with   $t_2=t_3^{h/\rho}.$

We can rapidly dispose of the case in which $s\le
P^3|\phi_3|$. If this happens then the second part of Lemma
\ref{lem:Weyl3}, together with Lemma~\ref{lem:W1}, yields
\[
1\ll P^{2+\ve}t_2^2|\phi_3|\ll
P^{2+\ve}t_3^{2n/(n-1)}.P^{\ve-3}t_3^8s^{-1},
\]
whence
\[1\ll P^{-2+4\ve}t_3^{16+4n/(n-1)}s^{-2}.\]
We now write $\cP_1$ for the set of $\al_3$ for which
$s\le P^3|\phi_3|$ and the value of $T_3$ lies in our dyadic range
$t_3<T_3\le 2t_3$.  Then if $\al_3\in\cP_1$ we have 
\[S(\alpha_3,\alpha_2)\ll P^nt_3^{-n}\ll 
P^{n-2+4\ve}t_3^{16+4n/(n-1)-n}s^{-2}.\] 
Moreover, since $\phi_3\ll P^{\ve-3}t_3^8s^{-1}$ we have
\[\int_{\cP_1}s^{-2}\d\al_3\ll\sum_{s\ll P^{\ve}t_3^8}s^{-2}\sum_{b_3\bmod{s}} 
P^{\ve-3}t_3^8s^{-1}\ll P^{\ve-3}t_3^8.\] 
The contribution to the minor arc integral is therefore
\[\ll P^{n-2+4\ve}t_3^{16+4n/(n-1)-n}.P^{2\ve-3}t_3^8.\]
In view of \eqref{eq:flight} this provides a satisfactory bound if
$n>24+4n/(n-1)$, so that $n\ge 29$ will be sufficient.  Thus in
applying Lemma \ref{lem:Weyl3} we may henceforth take
\beql{eq:sb}
s\gg P^{1-\ve}t_2^{-2}.
\eeq

\bigskip

We proceed to examine the case in which the
first term on the right in Lemma \ref{lem:main-poisson} dominates the
second.  Thus we will assume that
\[S(\al_3,\al_2)\ll q_0^{1/2}P^{\ve}q^{n/2}\Theta^{n/2}.\]
It follows from  \eqref{eq:Q2Q3}, \eqref{eq:al2al3} and
\eqref{eq:def-T} that $q\Theta\ll P^{5/3}$, whence
\beql{eq:sup1}
S(\al_3,\al_2)\ll q_0^{1/2}P^{5n/6+\ve}.
\eeq
We now consider the contribution 
from the set $\mathcal{P}_2$ of pairs $(\al_3,\al_2)$
for which \eqref{eq:sup1} holds and 
$t_3<T_3\le 2t_3$.  It will be convenient to write
$I_{t_3}(\minor)$ for the corresponding part of the minor arc integral.
From \eqref{eq:T3def} we will have
\beql{eq:sup2}
S(\al_3,\al_2)\ll P^nt_3^{-n}.
\eeq
As in \eqref{eq:al3meas}
the measure of the available set of points $\al_3$ is
$O(P^{2\ve-3}t_3^{16})$, whence
\begin{equation}\label{eq:measP2}
\meas(\cP_2)\ll P^{2\ve-3}t_3^{16}.
\end{equation}
Thus we certainly have
\beql{eq:me1}
I_{t_3}(\minor)\ll P^{2\ve-3}t_3^{16}.P^nt_3^{-n}.
\eeq
For an alternative bound we consider two cases.  We give ourselves a
parameter $Q_0\ge 1$, to be chosen shortly, and consider
separately the ranges $q_0\le Q_0$ and $q_0\ge Q_0$. If $q_0\le Q_0$,
the bounds \eqref{eq:sup1} and \eqref{eq:measP2} show that the
contribution to $I_{t_3}(\minor)$ is
\begin{equation}\label{eq:landing}
\ll P^{2\ve-3}t_3^{16}.Q_0^{1/2}P^{5n/6+\ve}.
\end{equation}
To investigate the second alternative we consider
the subset, $\cP_3$ say, of $\cP_2$ for
which the corresponding value of $q_0$ is at least $Q_0$.
Using the facts that $\al_3$ and $\al_2$ satisfy
\eqref{eq:al2al3}, and that $q_0$ is given by \eqref{eq:q02}, we have
\[\meas(\cP_3)\ll\sum_{\substack{q\le Q_3Q_2\\ q_0\ge Q_0}}\sum_{a_3}
\sum_{a_2}(qQ_3)^{-1}(qQ_2)^{-1}
\ll P^{-5/3}\sum_{\substack{q\le Q_3Q_2\\ q_0\ge Q_0}}\sum_{a_3}q^{-1}.\]
If we define 
\[q_3=\prod_{p\mid q_0}p\]
then $q_3\mid a_3$, so that there are $O(q/q_3)$ 
available values for $a_3$, and we deduce that
\[\meas(\cP_3)\ll P^{-5/3}\sum_{\substack{q\le Q_3Q_2\\ q_0\ge Q_0}}q_3^{-1}
\ll \sum_{q_0\ge Q_0}q_3^{-1}q_0^{-1}.\]
However a standard estimation via Rankin's method shows that
\begin{align*}
\sum_{q_0\ge Q_0}q_3^{-1}q_0^{-1}&\le Q_0^{\ve-1}\sum_{q_0\ge Q_0}q_3^{-1}q_0^{-\ve}\\
&\le Q_0^{\ve-1}\sum_{q_0=1}^{\infty}q_3^{-1}q_0^{-\ve}\\
&=Q_0^{\ve-1}\prod_p\left\{1+p^{-1-\ve}+p^{-1-2\ve}+p^{-1-3\ve}+\cdots\right\}\\
&\ll  Q_0^{\ve-1}.
\end{align*}
We therefore deduce that
\[\meas(\cP_3)\ll P^{2\ve}Q_0^{-1},\]
so that the estimate \eqref{eq:sup2} shows that the 
contribution to $I_{t_3}(\minor)$ is
\[\ll P^{2\ve}Q_0^{-1}.P^n t_3^{-n}.\]

Comparing this bound with \eqref{eq:landing} we see 
that we have an estimate
\[I_{t_3}(\minor)\ll
P^{2\ve-3}t_3^{16}.Q_0^{1/2}P^{5n/6+\ve}+P^{2\ve}Q_0^{-1}.P^n
t_3^{-n},\] 
for any $Q_0\ge 1$, covering both cases $q_0\le Q_0$ and $q_0\ge Q_0$.
We choose 
\[Q_0=1+P^{n/9+2}t_3^{-2n/3-32/3}\]
so as to balance the two terms approximately, and then deduce that 
\[I_{t_3}(\minor)\ll P^{3\ve}\{P^{5n/6-3}t_3^{16}+P^{8n/9-2}t_3^{(32-n)/3}\}.\]
We now combine this with \eqref{eq:me1} to produce
\begin{align*}
I_{t_3}(\minor)\ll P^{3\ve}\{&\min\big(P^{n-3}t_3^{16-n}\,,\,P^{5n/6-3}t_3^{16}\big)\\
&+\min\big(P^{n-3}t_3^{16-n}\,,\,P^{8n/9-2}t_3^{(32-n)/3}\big)\}.
\end{align*}
If $n\ge 16$ then we use the inequality 
$\min(A,B)\leq A^{16/n}B^{(n-16)/n}$ to conclude that
\begin{align*}
\min\big(P^{n-3}t_3^{16-n}\,,\,P^{5n/6-3}t_3^{16}\big)
&\le
P^{n-5-(n-28)/6}.
\end{align*}
If $16\le n\le 32$ 
we use $\min(A,B)\leq A^{(32-n)/(2n-16)}B^{(3n-48)/(2n-16)}$
to deduce that 
$$
\min\big(P^{n-3}t_3^{16-n}\,,\,P^{8n/9-2}t_3^{(32-n)/3}\big)
\leq P^{n-5-\eta}
$$
with
\[\eta=\frac{(n-29)(n-8)+8}{6(n-8)}.\]
These bounds make it clear that $I_{t_3}(\minor)\ll P^{n-5-\ve}$ for a
small $\ve>0$, when $29\le n\le 32$;
and if $n\ge 33$ then 
\begin{align*} 
\min\big(P^{n-3}t_3^{16-n}\,,\,P^{8n/9-2}t_3^{(32-n)/3}\big) 
&\leq P^{8n/9-2}t_3^{(32-n)/3}\\
&\leq P^{8n/9-2}\\ 
&\leq P^{n-5-2/3}. 
\end{align*}
We may therefore conclude as follows.
\begin{lemma}\label{lem:ft}
If $h=n$ and $\rho\ge n-1$ then the contribution to the minor arc
integral when the first term on the right in Lemma \ref{lem:main-poisson}
dominates the second, will be $o(P^{n-5})$, provided that $n\ge 29$.
\end{lemma}

\bigskip

We turn now to the pairs $(\al_3,\al_2)$ for which the second 
term on the right in Lemma \ref{lem:main-poisson}
dominates the first, so that
\[S(\alpha_3,\alpha_2)\ll  q_0^{1/2}P^{n+\ve}q^{-n/2}q_2^{n/3}.\]
We now recall that we may assume that $a_3/q=b_3/s$ with 
$\gcd(s,b_3)=1$.  Thus in particular we will have $s\mid q$.
In view of the definitions \eqref{eq:q02} we therefore deduce that 
\begin{align*}
q_0^{1/2}q^{-n/2}q_2^{n/3}&=q_0^{(1-n)/2}q_1^{-n/2}q_2^{-n/6}\\
&\le 
(q_0q_1q_2)^{-(n+4)/6}q_2^{2/3}\\
&=q^{-(n+4)/6}q_2^{2/3}\\
&\le s^{-(n+4)/6}q_2^{2/3}
\end{align*}
as soon as $n\ge 4$.
Moreover, if $p^e\|q_2$ then we have $p^e\mid qb_3=a_3s$ and $p^{1+v}\nmid a_3$,
whence $p^e\mid sp^v$. It follows that $q_2\mid Ds$, with 
\[D=2\prod_{i=1}^{n-1} d_i\] 
for the coefficients $d_i$ in \eqref{eq:didef}.   
Now let $\cP_4$ be the set of pairs $(\al_3,\al_2)$ for
which the corresponding values of $s, q_2$ and $T_3$ lie in given 
dyadic ranges $S< s\le 2S$, $Q_2<q_2\le 2Q_2$ and $t_3<T_3\le 2t_3$, so
that
\[S(\alpha_3,\alpha_2)\ll  P^{n+\ve}S^{-(n+4)/6}Q_2^{2/3}\]
on $\cP_4$.  Since $q_2\mid Ds$ and $D\ll 1$ there are $O(S/q_2)$ 
choices for $s$, given $q_2$. We have $s|\phi_3|\ll P^{\ve-3}t_3^8$ 
by Lemma \ref{lem:W1} and so we may calculate that 
\begin{align*}
\meas(\cP_4)&\ll \sum_{Q_2<q_2\le 2Q_2}\sum_{S<s\le
  2S}\sum_{b_3\bmod{s}}P^{\ve-3}s^{-1}t_3^{8}\\
&\ll \sum_{Q_2<q_2\le 2Q_2}P^{\ve-3}t_3^{8}Sq_2^{-1}\\
&\ll P^{\ve-3}t_3^{8}SQ_2^{-2/3},
\end{align*}
since $q_2$ runs over cube-full numbers. This yields the bound
\beql{eq:B2}
\int_{\cP_4}|S(\alpha_3,\alpha_2)|\d\al_3\d \al_2  
\ll  P^{n-3+\ve}t_3^8S^{-(n-2)/6}.
\eeq
Alternatively, \eqref{eq:sup2} produces
\[\int_{\cP_4}|S(\alpha_3,\alpha_2)|\d\al_3\d \al_2  
\ll P^{n}t_3^{-n}\meas(\cP_4)\ll P^{n-3+\ve}t_3^{8-n}S.\]
We may combine these to give
\begin{align*}
\int_{\cP_4}&|S(\alpha_3,\alpha_2)|\d\al_3\d \al_2\\  
&\ll P^{n-3+\ve}\min\left(t_3^8S^{-(n-2)/6}\,,\,t_3^{8-n}S\right)\\
&\ll P^{n-3+\ve}\left(t_3^8S^{-(n-2)/6}\right)^{6/(n+4)}
\left(t_3^{8-n}S\right)^{(n-2)/(n+4)}\\
&=P^{n-3+\ve}t_3^{-\kappa_1},
\end{align*}
with
\[\kappa_1=\frac{n^2-10n-32}{n+4}.\]

We may also couple \eqref{eq:B2} with \eqref{eq:sb} to produce a bound
\[\ll  P^{n-3-(n-2)/6+n\ve}t_3^8t_2^{(n-2)/3}\ll 
P^{n-3-(n-2)/6+n\ve}t_3^{\kappa_2},\]
with
\[\kappa_2= 8+\frac{(n-2)n}{3n-3}.\] 
Comparing this with the previous bound we deduce that
\begin{align*}
\int_{\cP_2}|
&S(\alpha_3,\alpha_2)|\d\al_3\d \al_2\\  
&\ll P^{n\ve}\min\left(P^{n-3}t_3^{-\kappa_1}\,,\,
P^{n-3-(n-2)/6}t_3^{\kappa_2}\right)\\
&\le P^{n\ve}\min\left(P^{n-3}t_3^{-\kappa_1}\right)^{(n-14)/(n-2)}
\left(P^{n-3-(n-2)/6}t_3^{\kappa_2}\right)^{12/(n-2)}\\
&=P^{n-5+n\ve}t_3^{-\kappa}
\end{align*}
with
\begin{align*}
\kappa&=\frac{n-14}{n-2}\kappa_1-
\frac{12}{n-2}\kappa_2\\
&=\left(\frac{n-14}{n-2}\right)\left(\frac{n^2-10n-32}{n+4}\right)-
\frac{12}{n-2}\left(8+\frac{(n-2)n}{3n-3}\right).
\end{align*}
A slightly unpleasant calculation confirms that $\kappa>0$ whenever
$n\ge 28$.  This completes our treatment of the minor arcs, which we
summarize as follows.
\begin{lemma}\label{lem:minarcest}
If $h=n$ and $\rho\ge n-1$ then the claimed minor arc estimate
\eqref{eq:minor} holds as soon as $n\ge 29$.
\end{lemma}
This will suffice for our application to Theorem \ref{main}.

\section{Major arc contribution}\label{s:5}

The purpose of this section is to
complete the proof of Theorems \ref{main'} and \ref{main}, by 
establishing
\eqref{eq:major} under suitable hypotheses on $\major$ and the forms
$C$ and $Q$.  
In what follows we will put $h=n$ if $C$ is non-singular and $h=h(C)$
otherwise. Moreover, 
we continue to adopt the notation 
$$
\rho=\rank(Q),
$$
for the rank of the quadratic form $Q$. By Corollary \ref{c:3.2} we
have $\rho\geq n-1$ when the intersection $C=Q=0$ is non-singular.

It is now time to reveal the weight functions $\omega$ that we shall
use in the definition \eqref{eq:def-N} of our counting function 
$$
N_{\omega}(X;P)=\sum_{\substack{\x\in \ZZ^n\\ C(\x)=Q(\x)=0}} \omega(\x/P).
$$
There is nothing to prove unless the variety $X$ contains a
non-singular real point.  
Consequently, we
let $\x_0\in \RR^n$ be a non-zero vector such that $C(\x_0)=Q(\x_0)=0$ and
$\nabla C (\x_0)$ is not proportional to $\nabla Q(\x_0)$. 
We will find it convenient to work with a weight function 
that forces us to count points lying very close to $\x_0$. 
For any $\xi\in (0,1]$, we define the function $\omega: \RR^n 
\rightarrow 
\RR_{\geq 0}$ by 
$$
\omega(\x):=\nu\left(\xi^{-1}\|\x-\x_0\|\right),
$$
where $\|\y\|=\sqrt{y_1^2+\cdots+y_n^2}$ and 
$$
\nu(x)=\begin{cases}
e^{-1/(1-x^2)}, & \mbox{if $|x|<1$},\\
0, & \mbox{if $|x|\geq 1$}.
\end{cases}
$$
We will require  $\xi$ to be 
sufficiently small, with  $1\ll \xi \leq 1$.
It is clear that $\omega$ is infinitely
differentiable, and that it is supported on the region
$|\x-\x_0| \leq \xi$.   
Moreover, there exist
constants $c_j>0$ depending only on $j$ and $\xi$ such that
$$
\max\Big\{ \Big|
\frac{\partial^{j_1+\cdots+j_n}\omega (\x)}{\partial^{j_1}x_1\cdots
\partial^{j_n}x_n}\Big|: ~\x \in \RR^n, ~j_1+\cdots+j_n=j\Big\}\leq c_j,
$$
for each integer $j \geq 0$.

We are now ready to begin our  analysis of the exponential sums 
$S(\al_3,\al_2)$ on the set of major arcs
$\mathfrak{M}$  defined in \S \ref{s:2},  for $\delta\in (0,\frac{1}{3}).$
Let us define 
$$
S(\a,q):=\sum_{\y\bmod{q}}e_q\big(a_3C(\y)+a_2Q(\y)\big),
$$
for $\a=(a_3,a_2)$ with $\gcd(q,\a)=1$.
Our work in this section will lead us to study the truncated singular series
\begin{equation}\label{21-sing}
\ss(R)=\sum_{q\leq R} \frac{1}{q^n}\sum_{\substack{\a \bmod{q}\\ 
\gcd(q,\a)=1}}
S(\a,q),
\end{equation}
for any $R>1$. 
We put  $\ss=\lim_{R\rightarrow \infty}\ss(R)$, whenever this limit exists.
Next, let
\begin{equation}
  \label{21-si}
\mathfrak{I}(R)=\int_{-R}^{R} \int_{-R}^{R} 
\int_{\RR^n} \omega(\x)
e\big(\gamma_3 C(\x)+\gamma_2 Q(\x)\big)\d\x\d \gamma_3\d \gamma_2,
\end{equation}
for any $R>0$. We put  $\mathfrak{I}=\lim_{R\rightarrow \infty}
\mathfrak{I}(R)$, whenever the limit  exists.  
The main aim of this section is to establish the
following result.

\begin{lemma}\label{lem:major}
Assume that $(h-24)(\rho-4)>96$.   Then the singular series 
$\ss$ and the singular integral $\mathfrak{I}$ are absolutely 
convergent.  Moreover, if we choose
\beql{eq:choosedelta}
\delta=1/7 
\eeq
then there is a positive constant $\Delta$ such that
$$
\iint_{\major}S(\alpha_3,\alpha_2)\d\alpha_3\d\alpha_2=\ss \mathfrak{I}P^{n-5}
+O(P^{n-5-\Delta}).
$$
\end{lemma}

Taking the statement of Lemma \ref{lem:major} on faith, let us
indicate how it can be used to complete the proof of Theorems
\ref{main'} and \ref{main}. 
In the context of Theorem \ref{main}, for which $h=n\geq 29$ and
$\rho\geq n-1$, we combine Lemma \ref{lem:minarcest} and Lemma
\ref{lem:major} to deduce that  
$$
N_{\omega}(X;P)=\ss\mathfrak{I}P^{n-5}+o(P^{n-5}), 
$$
as $P\rightarrow \infty$, with both 
$\ss$ and  $\mathfrak{I}$  absolutely 
convergent.   The same asymptotic formula holds when
$(h(C)-32)(\rho-4)>128$, as in Theorem~\ref{main'}. 
This  can be seen by combining  Lemma \ref{lem:Weyl6} with
Lemma~\ref{lem:major}.   

In order to complete the proof of Theorems \ref{main'} and \ref{main}
we need to show that  
$\ss \mathfrak{I}>0$ whenever $\Xns(\mathbf{A})\neq \emptyset$.
Indeed, if $\ss \mathfrak{I}>0$ for any $[\x_0]\in\Xns(\RR)$ 
then it will follow from our asymptotic formula 
for $N_{\omega}(X;P)$ that $X(\QQ)$ is Zariski-dense in $X$, 
whence the existence of a point in $\Xns(\QQ)$ is assured. 
The proof that $\ss>0$ follows a standard line of reasoning, as in
\cite[Lemma~7.1]{birch}, and makes use of the fact that $\ss$ is 
absolutely convergent.
To show that $\mathfrak{I}>0$, it will suffice to 
show that $\mathfrak{I}(R)\gg 1$ for sufficiently large values of $R$.
This again is standard and will follow from an easy adaptation 
of work of Heath-Brown 
\cite[\S 10] {hb-10} on the corresponding problem for a single cubic form.
The only difference lies in the choice of weights used
and the fact that we now have a complete intersection of codimension $2$,
but neither of these alters the
nature of the proof.  
Performing the integrations over $\gamma_3$ and $\gamma_2$, and writing 
$\x=\x_0+\y$, it follows from \eqref{21-si} that
\begin{align*}
\mathfrak{I}(R)
&=\int_{\RR^n} \omega(\x) \frac{\sin(2\pi R C(\x))   \sin(2\pi R
  Q(\x))}{\pi^2 C(\x)Q(\x)} 
\d\x\\
&=\int_{\RR^n} \nu(\xi^{-1}\|\y\|)  
\frac{\sin(2\pi R C(\x_0+\y))\sin(2\pi R Q(\x_0+\y))}{\pi^2 C(\x_0+\y)Q(\x_0+\y)}
\d\y.
\end{align*}
Let $a_i=\partial C/\partial x_i(\x_0)$ and 
$b_i=\partial Q/\partial x_i(\x_0)$ for $1\leq i\leq n$. 
We may assume 
 without loss of generality that 
$a_1b_2-a_2b_1\neq 0$. 
The need for $\xi>0$ to be sufficiently small 
emerges through an application of the inverse function
theorem. Since $|\y|\leq \xi$, if we write   
\begin{align*}
z_3&=C(\x_0+\y)=a_1y_1+\cdots+a_ny_n +P_2(\y)+P_3(\y),\\
z_2&=Q(\x_0+\y)=b_1y_1+\cdots+b_ny_n +Q_2(\y),
\end{align*}
for forms $P_i$ of degree $i$ and $Q_2$ of degree $2$, 
then $z_3,z_2\ll \xi$ and we can 
invert this expression to represent $y_1$ and $y_2$ as a power series in 
$z_3,z_2,y_3,\ldots,y_n$, if $\xi$ is sufficiently small. We refer the
reader to \cite{hb-10} 
for the remainder of the argument.

\bigskip

To prove  Lemma \ref{lem:major} we begin by recalling that 
$q\leq P^{\delta}$, that we have $\a=(a_3,a_2)$ with $\gcd(q,\a)=1$, and that
$(\al_3,\alpha_2)\in \major_{\a,q}$, with 
$\al_i=a_i/q+\theta_i,$ for $i=3,2$.  We will use the argument of
\cite[Lemma~5.1]{birch} to show that
\begin{equation}
  \label{eq:train1}
S(\al_3,\alpha_2)=q^{-n}P^n S(\a,q) I(\theta_3 P^3,\theta_2
P^2;\ma{0}) +O(P^{n-1+\delta}),   
\end{equation}
where $S(\a,q)$ is given above and $I$ is given by \eqref{eq:Iq}.

To see this we write $\x=\y+q\z$ in \eqref{eq:S}, where $\y$
runs over a complete set of residues modulo $q$, giving
\begin{equation}
  \label{eq:train2}
S(\al_3,\al_2)=
\sum_{\y \bmod{q}}e_q\big(a_3C(\y)+a_2Q(\y)\big)
\sum_{\z\in\ZZ^n}f(\z),
\end{equation}
with
$$
f(\z)=\omega\left(\frac{\y+q\z}{P}\right)
e\left(\theta_3 C(\y+q\z)+\theta_2 Q(\y+q\z)\right).
$$
We now want to replace the summation over $\z$ by an integration. 
If $\t\in [0,1]^n$ then 
\[f(\z+\t)=f(\z)+O(\max_{\ma{u}\in [0,1]^n} |\nabla f(\z+\ma{u})|).\]
Hence
\begin{align*}
\Big|\int_{\RR^n}f(\z)\d\z-\sum_{\z\in\ZZ^n}f(\z)\Big|
&\ll \meas (\mathcal{B})
\max_{\z\in \mathcal{B}}|\nabla f(\z)|\\
&\ll \Big(\frac{P}{q}\Big)^n \big(q/P+q|\theta_3|P^2+q|\theta_2|P)\\
&=
q^{1-n}P^{n-1}+|\theta_3|q^{1-n}P^{n+2}+|\theta_2|q^{1-n}P^{n+1},
\end{align*}
where $\mathcal{B}$ is an $n$-dimensional cube with sides of order
$1+P/q\leq 2P/q$. 
Substituting this
into \eqref{eq:train2}
and  making the change of variables $P\u=\y+q\z$, 
 we therefore deduce that
\begin{equation}
  \label{eq:train3}
\begin{split}
  S(\al_3,\alpha_2)=~&q^{-n}P^n S(\a,q) I(\theta_3 P^3,\theta_2 P^2;\ma{0}) 
 \\ 
&  +O(qP^{n-1}+|\theta_3|qP^{n+2}+|\theta_2|qP^{n+1}).
\end{split}
\end{equation}
This completes the proof of \eqref{eq:train1}, since
$|\theta_i|\leq P^{-i+\delta}$ and $q\leq P^\delta$ on the major 
arcs. 

Using \eqref{eq:train1}, and noting that the major arcs have measure
$O(P^{-5+5\delta})$, it is now easy to deduce that 
\begin{equation}\label{21-maj}
\iint_{\major} S(\al_3,\alpha_2)\d\al_3\d\alpha_2=
P^{n-5}\mathfrak{S}(P^\delta)\mathfrak{I}(P^\delta) +O(
P^{n-6+6\delta}), 
\end{equation}
where $\mathfrak{S}(P^\delta)$ is given by \eqref{21-sing}, and
$\mathfrak{I}(P^\delta)$ is given by \eqref{21-si}.

We proceed to use \eqref{eq:train1} in conjunction with our Weyl
estimates, Lemmas \ref{lem:W1} and \ref{lem:Weyl3}, to bound
$S(\a,q)$, with the aim of proving the following result.
\begin{lemma}\label{lem:completeWeyl}
Let $\ve>0$ be given. 
If $h$ and $\rho$ are both  positive then
\beql{eq:Saq}
S(\a,q)\ll q^{n+\ve}\left(\frac{q}{\gcd(q,a_3)}\right)^{-h/8}
\eeq
and
\beql{eq:Saq'}
S(\a,q)\ll q^{n+\ve}\gcd(q,a_3)^{-\rho/2}.
\eeq
\end{lemma}

\begin{proof}
To prove this, we reverse our normal point of view, and
think of $q$ as given and of $P$ as being large in terms of $q$.
Specifically it will suffice to take
\beql{eq:Plb}
P=q^{8n}.
\eeq
When
$\theta_3=\theta_2=0$ we have $I(0,0;\ma{0})\gg 1$, whence
\eqref{eq:train1} yields
\begin{equation}\begin{split}\label{eq:STb}
S(\a,q)&\ll 1+q^nP^{-n}|S(a_3/q,a_2/q)|\\
&=1+q^nT_3^{-h}=1+q^nT_2^{-\rho},
\end{split}\end{equation}
since \eqref{eq:Plb} shows that $P^{1-\delta}\gg q^n$.

We proceed to apply Lemma
\ref{lem:W1}, bearing in mind that the integer $s$ is not necessarily
equal to $q$.  Thus we have $a_3/q=b_3/s+\phi_3$ and
\beql{eq:st3}
s(1+P^3|\phi_3|)\ll P^{\ve}T_3^8.
\eeq
If $a_3/q\not=b_3/s$ then $|\phi_3|\ge (sq)^{-1}$, whence
\[T_3^8\gg P^{3-\ve}s|\phi_3|\ge P^{3-\ve}q^{-1}.\]
Then, taking $\ve<1$, we see that \eqref{eq:STb} leads to the estimate
\[S(\a,q)\ll 1+q^{n}T_3^{-1}\ll 1+q^n.P^{(\ve-3)/8}q^{1/8}\ll
1+q^{n+1}P^{-1/4}\ll 1\]
in view of \eqref{eq:Plb}. This is more than sufficient for the
lemma, and so we henceforth assume that $a_3/q=b_3/s$ and
that $\phi_3=0$. Thus $sa_3=qb_3$ with $\gcd(s,b_3)=1$, whence 
$s=q/\gcd(q,a_3)$.  Moreover \eqref{eq:st3} reduces to $s\ll
P^{\ve}T_3^8$, so that $T_3\gg P^{-\ve/8}(q/\gcd(q,a_3))^{1/8}$.
Inserting this into \eqref{eq:STb} leads to the estimate
\[S(\a,q)\ll 1+q^n.P^{\ve h/8}\left(\frac{q}{\gcd(q,a_3)}\right)^{-h/8}.\]
This is suitable for \eqref{eq:Saq}, given our choice
\eqref{eq:Plb}, on re-defining $\ve$.

To obtain \eqref{eq:Saq'} we apply Lemma \ref{lem:Weyl3} which 
either shows that
\[T_2^2\gg P^{1-\ve}s^{-1}\ge P^{1-\ve}q^{-1},\]
or produces a positive integer $u\ll T_2^2$ for which
\beql{eq:a2}
\|sua_2/q\|\ll P^{-2+\ve}sT_2^2.
\eeq
If the first alternative holds then, taking $\ve<1/2$, 
we find that \eqref{eq:STb} produces a bound
\[S(\a,q)\ll 1+q^{n}T_2^{-1}\ll 1+q^n.P^{(\ve-1)/2}q^{1/2}\ll
1+q^{n+1}P^{-1/4}\ll 1,\]
in view of \eqref{eq:Plb}. Again, this is more than sufficient for the
lemma, and so we examine the second alternative.

If  $q\nmid sua_2$ the bound \eqref{eq:a2} would imply that $q^{-1}\ll
P^{-2+\ve}sT_2^2$, so that
\[T_2^2\gg P^{2-\ve}(sq)^{-1}\ge P^{2-\ve}q^{-2}.\]
Just as above this would produce an acceptable estimate
\[S(\a,q)\ll 1+q^{n}T_2^{-1}\ll 1+q^n.P^{(\ve-2)/2}q\ll
1+q^{n+1}P^{-1/4}\ll 1.\]
On the other hand, if $q\mid sua_2$ then $\gcd(q,a_3)\mid ua_2$, since
we have $s=q/\gcd(q,a_3)$, as noted above. Recalling that $\gcd(q,a_3,a_2)=1$
we deduce that $\gcd(q,a_3)\mid u$, so that
$\gcd(q,a_3)\le u\ll T_2^2$.  Thus \eqref{eq:STb} yields
\[S(\a,q)\ll 1+q^n\gcd(q,a_3)^{-\rho/2},\]
as required for \eqref{eq:Saq'}.  This completes the proof of the lemma.
\end{proof}

We can now handle the singular series. 
Let
$$
A(q)=\sum_{\substack{\a
      \bmod{q}\\ \gcd(q,\a)=1}}|S(\a,q)|.
$$
Then 
we have
\begin{align*}
A(q)
&\ll
q\sum_{a_3\bmod{q}}q^{n+\ve}
\min\left(
  \left(\frac{q}{\gcd(q,a_3)}\right)^{-h/8}\,,\,\gcd(q,a_3)^{-\rho/2}\right). 
\end{align*}
There are at most $q/d$ values of $a_3$ for which $\gcd(q,a_3)=d$, and
each one contributes a total
\begin{align*}
&\ll q^{n+1+\ve}\min\left((q/d)^{-h/8}\,,\,d^{-\rho/2}\right)\\
&\ll q^{n+1+\ve}\left((q/d)^{-h/8}\right)^{(4\rho+8)/(4\rho+h)}
\left(d^{-\rho/2}\right)^{(h-8)/(4\rho+h)}\\
&= q^{n+1+\ve-\xi} d
\end{align*}
with
\[\xi=\frac{h(\rho+2)}{8\rho+2h}.\]
It follows that
\begin{align*}
A(q)
&\ll
\sum_{d\mid q}qd^{-1}.q^{n+1+\ve-\xi}d\ll
q^{n+2+2\ve-\xi},
\end{align*}
so that the singular series is absolutely convergent when $\xi>3$, and
\[\ss(R)=\ss+O(R^{2\ve-(\xi-3)}).\]
Since $\xi>3$ when $(h-24)(\rho-4)>96$ the claim in 
Lemma~\ref{lem:major} follows.

We now estimate the exponential integral 
$I(\bga;\0)$, for general values of $\bga=(\gamma_3,\gamma_2)$.

\begin{lemma}\label{21-int1}
We have $I(\bga;\0)\ll 1$ for any $\bga$.  Moreover if $h$ and $\rho$
are positive, and if $\ve\in (0,1/8)$, then 
\beql{eq:I1}
I(\bga;\0)\ll |\gamma_3|^{-h/8}|\bga|^{\ve}
\eeq
and
\beql{eq:I2}
I(\bga;\0)\ll
\left(\frac{|\gamma_2|}{1+|\gamma_3|}\right)^{-\rho/2}|\bga|^{\ve}. 
\eeq
\end{lemma}

\begin{proof}
The estimate $I(\bga;\ma{0})\ll 1$ is trivial. Moreover it implies
both \eqref{eq:I1} and \eqref{eq:I2} when $|\bga|\le 1$.
We assume  henceforth that $|\bga|>1$, and follow an argument
analogous to that used for Lemma \ref{lem:completeWeyl}

Taking $a_3=a_2=0$ and $q=1$ in \eqref{eq:train3}, and setting
$\alpha_3=P^{-3}\gamma_3$ and $\al_2=P^{-2}\gamma_2$, we deduce that
\begin{align*}
I(\bga;\ma{0})&=P^{-n}S(\al_3,\alpha_2) +O(P^{-1}|\bga|)\\
&=T_3^{-h}+O(P^{-1}|\bga|)\\
&=T_2^{-\rho}+O(P^{-1}|\bga|),
\end{align*}
for any $P\geq 1$. We will choose $P$ to be large, given by
\beql{eq:Pg}
P=|\bga|^{2n(2n+8)},
\eeq
so that
\beql{eq:lounge}
I(\bga;\ma{0})=T_3^{-h}+O(|\bga|^{-n})=T_2^{-\rho}+O(|\bga|^{-n}).
\eeq

We now need estimates for the quantities $T_3$ and $T_2$.
We begin by applying Lemma~\ref{lem:W1}, which shows that
\[P^{-3}\gamma_3=\al_3=\frac{b_3}{s}+\phi_3\]
with
\beql{eq:again}
s(1+P^3|\phi_3|)\ll P^{\ve}T_3^8.
\eeq
If $b_3\not=0$ then
\[s^{-1}\le |b_3|s^{-1}\le P^{-3}|\gamma_3|+|\phi_3|\le 
P^{-3}|\bga|(1+P^3|\phi_3|).\]
It follows that $s(1+P^3|\phi_3|)|\bga|\ge P^3$, whence
\[T_3^8\gg P^{3-\ve}|\bga|^{-1}\gg|\bga|^{8n}\]
for $\ve<1$, in view of our choice \eqref{eq:Pg} of $P$. The estimates
\eqref{eq:I1} and \eqref{eq:I2} then follow from 
\eqref{eq:lounge}, for the case $b_3\not=0$.

We therefore assume that $b_3=0$ and hence that
$\phi_3=P^{-3}\gamma_3$.  To prove \eqref{eq:I1} we observe that
\eqref{eq:again} yields
\[T_3^8\gg P^{-\ve}s|\gamma_3|\gg P^{-\ve}|\gamma_3|.\]
Inserting this
into \eqref{eq:lounge} leads to the bound
\[I(\bga;\ma{0})\ll P^{h\ve}|\gamma_3|^{-h/8}+|\bga|^{-n}.\]
The relation \eqref{eq:Pg} allows us to replace $P^{h\ve}$ by
$|\bga|^{\ve}$ on re-defining $\ve$, and \eqref{eq:I1} follows.

We turn now to the estimate \eqref{eq:I2}, for which we use
Lemma \ref{lem:Weyl3}.  This tells us that either 
\[T_2^2\gg\frac{P^{1-\ve}}{s+P^3|\phi_3|}\]
or that there is a positive integer $u\ll T_2^2$ for which
\beql{eq:alt2}
\|su\al_2\|\ll P^{-2+\ve}s(1+P^3|\phi_3|)T_2^2.
\eeq
In the first case we have
\[P^{1-\ve}\ll T_2^2(s+P^3|\phi_3|)\ll T_2^2.P^{\ve}T_3^8\]
by \eqref{eq:again}. Thus if $\ve<1/4$ we will have
\[P^{1/2}\ll T_2^2T_3^8=T_3^{2h/\rho+8}\le T_3^{2n+8}.\]
Our choice \eqref{eq:Pg} then shows that $T_3\ge|\bga|^n$, so that
\eqref{eq:I2} follows from \eqref{eq:lounge}.

If the second alternative \eqref{eq:alt2} holds we can write
\[\al_2=\frac{b_2}{su}+\phi_2\]
with
\beql{eq:ph2}
\phi_2\ll u^{-1}P^{-2+\ve}(1+P^3|\phi_3|)T_2^2.
\eeq
If $b_2\not=0$ then
\[(su)^{-1}\le |b_2|(su)^{-1}\le P^{-2}|\gamma_2|+|\phi_2|,\]
so that \eqref{eq:again} yields
\[P^2\ll su|\gamma_2|+suP^2|\phi_2|
\ll P^{\ve}T_3^8.T_2^2|\bga|+P^{\ve}T_3^8T_2^2.\]
This produces $P^2\ll PT_3^8T_2^2$ on taking $\ve<1/2$ and using the 
crude bound $|\bga|\le P^{1/2}$ from \eqref{eq:Pg}.  We can then deduce
\eqref{eq:I2} just as in the previous paragraph.

We are left with the case in which $b_2=0$, so that
$P^{-2}\gamma_2=\al_2=\phi_2$. Since $\phi_3=P^{-3}\gamma_3$ it
follows from \eqref{eq:ph2} that
\[\gamma_2\ll P^{\ve}(1+|\gamma_3|)T_2^2.\]
Thus \eqref{eq:lounge} produces
\[I(\bga;\ma{0})\ll P^{\ve\rho/2}
\left(\frac{|\gamma_2|}{1+|\gamma_3|}\right)^{-\rho/2}+|\bga|^{-n}.\]
The first term on the right dominates the second, and we may replace
$P^{\ve\rho/2}$ by $|\bga|^{\ve}$ after re-defining $\ve$, in view of
our choice \eqref{eq:Pg} of $P$.  This establishes \eqref{eq:I2},
thereby completing our treatment of Lemma \ref{21-int1}.
\end{proof}

We are now ready to show that the singular integral converges.  We have
\beql{eq:Itail}
\mathfrak{I}-\mathfrak{I}(R)=\iint_{|\bga|\geq R} I(\bga;\ma{0})
\d\bga
\eeq
and we split the region of integration into two parts, to use the two
estimates of Lemma \ref{21-int1}. When 
$|\gamma_2|\le|\gamma_3|^{1+h/(4\rho)}$
and $|\bga|\ge R$ we have
\[I(\bga;\ma{0})\ll |\gamma_3|^{-h/8+\ve}\]
and 
\[|\gamma_3|\ge R^{4\rho/(h+4\rho)}.\]
The corresponding contribution to \eqref{eq:Itail} is then
\[\ll \int_{R^{4\rho/(h+4\rho)}}^{\infty}x^{1+h/(4\rho)}x^{-h/8+\ve}\d x
\ll R^{-\mu+\ve},\]
with
\[\mu=\frac{h\rho-16\rho-2h}{2h+8\rho}.\]
Similarly, when
$|\gamma_2|\ge|\gamma_3|^{1+h/(4\rho)}$
and $|\bga|\ge R$ we have
\[I(\bga;\ma{0})\ll (1+|\gamma_3|)^{\rho/2}|\gamma_2|^{-\rho/2+\ve}\]
and 
\[|\gamma_2|\ge R.\]
In this case the contribution to \eqref{eq:Itail} is 
\begin{align*}
&\ll \int_R^{\infty}x^{-\rho/2+\ve}
\int_0^{x^{4\rho/(h+4\rho)}}(1+y)^{\rho/2}\d y\, \d x\\
&\ll \int_R^{\infty}x^{-\rho/2+\ve}x^{(1+\rho/2)4\rho/(h+4\rho)}\d x\\
&\ll R^{-\mu+\ve},
\end{align*}
with the same $\mu$ as before. Thus we have absolute convergence when
$\mu>0$, or equivalently when $(h-16)(\rho-2)>32$. This suffices for
Lemma \ref{lem:major}.

To complete the proof of the lemma it remains to show that we can
replace the truncated singular series and integral in \eqref{21-maj}
by their limits, with an acceptable error.  This is clear however
since we have shown that $\ss$ and $\mathfrak{I}$ are finite, and
differ from $\ss(R)$ and $\mathfrak{I}(R)$ respectively by negative powers
of $R$.

\section{Proof of Theorem \ref{small_h}}\label{s:10}

In this section we will establish Theorem \ref{small_h}, subject to
various lemmas, all of which we will delay proving until the next
section. These will involve the parameters $n$,
$\rho=\rank(Q)$, $\qorder(C)$ and the $h$-invariants $h(C)$ and $h_Q(C)$.  
The latter, in particular, satisfy the inequalities 
$h_Q(C)\le h(C)\le h_Q(C)+1$,  as recorded in \eqref{h-hq}.
The reader should
note that in Theorem \ref{main'} one can replace $C$ by $C+LQ$ for a 
generic $L$, and
hence that it is the maximal value of $h(C+LQ)$ which is of relevance there.

We begin by recording some basic deductions about the above
parameters. We may assume that 
\beql{rholb}
\rho\ge n-13\ge 36,
\eeq
because otherwise 
$Q$ vanishes on a
$\QQ$-rational 13-plane, and we can conclude as in the 
proof of Theorem \ref{elementary1}. We will always have
$h_Q(C)\le \qorder(C)$, and indeed 
\beql{hQb}
h_Q(C)\le \qorder(C)-1,\;\;\;\mbox{if}\;\;\;\qorder(C)\ge 14.
\eeq
To see this,  suppose that $C=C(x_1,\ldots,x_m)$ with
$m=\qorder(C)\ge 14$, after a suitable change of variable.  Then,
by the result of Heath-Brown \cite{14} the form $C$ has a non-trivial
rational zero, which we may take to be $(0,\ldots,0,1)$. We can then write
$$C=x_1Q_1(x_1,\ldots,x_m)+\cdots+x_{m-1}Q_{m-1}(x_1,\ldots,x_m),$$
which shows that $h_Q(C)\le m-1$.

We may also eliminate the case in which $h_Q(C)=1$, which would mean
that one could take $C$ to factor as $LQ'$, say, over $\QQ$. If this
were to happen, then a smooth real point on $C=Q=0$ would lie either
on $Q=L=0$ or $Q=Q'=0$.  In the first case the Hasse--Minkowski theorem
suffices to complete the proof, since $n\ge 49\ge 6$. In the second 
case we apply Lemma \ref{SD} below, using the fact that $n\ge 49\ge
9$.  If some combination $aQ+bQ'$ were to have rank at most 4, then
$b\not=0$ by \eqref{rholb}.  However $bC+aLQ=L(aQ+bQ')$ would have
order at most 5, giving $\qorder(C)\le 5$ in contradiction to our
hypotheses. Thus the conditions needed for Lemma \ref{SD} do indeed
hold.  In what follows we will therefore be able to assume that
$h_Q(C)\ge 2$, and hence, via Lemma~\ref{lem:irred}, that $X$ is
absolutely irreducible.

Our strategy for the proof of
Theorem \ref{small_h} is now to combine two basic arguments, one of which
covers the case in which $h_Q(C)\le n-13$ and the other which deals
with larger values of $h_Q(C)$. 
We begin by discussing the second of these, which is more
straightforward.  The idea is to apply Theorem \ref{main'}, which will
require us to have smooth solutions for every completion of
$\QQ$. A smooth real solution is provided by our hypothesis, and we
will then require the following lemma to give us suitable
$p$-adic solutions.
\begin{lemma}\label{smoothp}
If $\qorder(C)\ge 4$, $h_Q(C)\geq 2$ 
and $\rho\ge 23$ then we have 
$\Xns(\QQ_p)\neq \emptyset$ for every prime $p$.
\end{lemma}

We will prove this in the next section.  
The conditions given are sufficient for our purposes but are probably 
not optimal.

The conditions of the lemma are amply met, in view of \eqref{rholb},
and Theorem \ref{main'} completes the argument if $h_Q(C)\ge n-12$,
since then 
\[(h_Q(C)-32)(\rho-4)\ge (n-44)(n-17)\ge 5\times 32>128,\] 
via a further application of \eqref{rholb}.

We will henceforth assume that $h_Q(C)\le n-13$ and we will replace $C$
by $C+LQ$ so that $h_Q(C)=h(C)=h$, say. Then after a suitable 
non-singular linear change of
variables, we can write
\[\mathbf{x}=(x_1, \ldots, x_n) = (u_1, \ldots, u_h;v_1, \ldots, v_s)\]
where $s=n-h$, so that $C$ and $Q$ take the shapes
\beql{ABRS*}
C(\x)=A(\u)+\sum_{j=1}^s v_jD_j(\u)+\sum_{i=1}^h u_i B_i(\v)
\eeq
and
\[ Q(\x) = R(\u;\v)+S(\v).\]
Here $A(\u)$ is a cubic form, while $D_j(\u)$, 
$B_i(\v)$, $R(\u;\v)$ and $S(\v)$ are quadratic forms, such that
$R(\u;\v)$ contains no quadratic terms in $\v$. 
We remark at once
that if
$\rank(S)<n-h$ then there is a vector
$\v_0\in\QQ^{s}-\{\mathbf{0}\}$ such that $S(\v_0)=0$.  Thus
$C(\mathbf{0},\v_0)=Q(\mathbf{0},\v_0)=0$, so that our system has a
nontrivial rational zero.  We may therefore assume that $\rank(S)=n-h$
from now on.  We can then apply a suitable linear transformation so as
to reduce $Q(\x)$ to the form
\[ Q(\x) = R(\u)+S(\v),\]
while leaving $C(\x)$ in the shape \eqref{ABRS*}, but with 
new forms $A$, $D_j$ and $B_i$. 

In what follows it will
also be useful to adopt the notation
\begin{equation}\label{ABRS}  
\begin{split}
C_\mathbf{a}(t, v_1, \ldots, v_s) &:=A(\mathbf{a})t^2+
\left\{\sum_{j=1}^sD_j(\mathbf{a})v_j\right\}t+
\sum_{i=1}^h a_iB_i(\v),\\
Q_\mathbf{a}(t,v_1, \ldots, v_s)&:=Q(t\mathbf{a},\v)=
R(\a)t^2+S(\v). 
\end{split}
\end{equation}
Note that both $Q_\mathbf{a}$ and $C_\mathbf{a}$ are quadratic forms
in $t, v_1, \ldots, v_s$. If we can show that there is a non-zero vector
$\mathbf{a}\in\QQ^h$ such that the forms $Q_{\mathbf{a}}$ and $C_{\mathbf{a}}$
have a common rational zero $(t_0,\v_0)$, then $C$ and $Q$ will have the
common zero $(t_0\,\mathbf{a},\v_0)$, which will complete the  the proof.

Here we will employ the following result, which will be an easy
corollary of Theorem A of Colliot-Th\'el\`ene, Sansuc and 
Swinnerton-Dyer \cite{CTSSD}.

\begin{lemma}\label{SD}
Let $f,g$ be quadratic forms over the rationals in $m\ge 9$ variables,
and suppose that the equations $f=g=0$ have a smooth 
solution over $\RR$, and that every form in the rational pencil has
rank at least $5$. 
 Then the forms have a common rational zero.
\end{lemma}
Note that in applying
Lemma \ref{SD} we will have forms in $s+1$ variables, where 
\[s+1=n-h+1\ge 14.\]
We call a non-zero real vector $\mathbf{a} \in \RR^h$ 
\emph{good} if the system of equations
$$
  Q_{\mathbf{a}}(t, v_1, \ldots, v_s) =
  C_{\mathbf{a}}(t, v_1, \ldots, v_s) = 0
$$
has a non-singular real zero. 
We shall then prove the following result.

\begin{lemma}\label{good}
If $n-h\ge 5$ and $\qorder(C)\ge\max(h+1,4)$ then
the set of $[\mathbf{a}] \in \PP^{h-1}(\QQ)$ 
such that $\mathbf{a}$ is
good, is Zariski-dense.
\end{lemma}

In our case we have $n-h\ge 13$ and 
\[\qorder(C)\ge \max(h+1,17),\] 
by \eqref{hQb} and the hypotheses of Theorem \ref{small_h}.

If there is any good rational $\mathbf{a}$ for which every
form in the rational pencil generated by $Q_{\mathbf{a}}$ 
and $C_{\mathbf{a}}$ has
rank at least 5, then Theorem \ref{small_h} will follow from Lemma
\ref{SD}.  We therefore proceed on the alternative assumption that for 
every good $\mathbf{a}\in\QQ^r$ there is a form
$\al C_{\mathbf{a}}+\beta Q_{\mathbf{a}}$ with 
$(\al,\beta)\in\QQ^2-\{(0,0)\}$, having rank
at most 4. We will prove the following lemma.

\begin{lemma}\label{splin}
Suppose that $n-h\ge 13$ and that there is a Zariski-dense set of
$[\mathbf{a}]\in\PP^{h-1}(\QQ)$ 
for each of which there is a form
$\al C_{\mathbf{a}}+\beta Q_{\mathbf{a}}$ with 
$(\al,\beta)\in\QQ^2-\{(0,0)\}$, having rank
at most $4$. Then after replacing $C$ by $C+LQ$ for a suitable
linear form $L$ defined over $\QQ$, and after making a suitable linear
change of variables, we may write $C(\x)$ in the shape
\[C(\x)=C(\u,\v)=\sum_{1\le i\le j\le H}u_i u_jL_{ij}(\u,\v),\]
with linear forms 
$L_{ij}$ defined over $\QQ$, and with $H=h+4$.
\end{lemma}

The reader should notice that our vectors $\u$ and $\v$ now have
different lengths from before.

We now define $Q_{\mathbf{a}}(t,\v)$ as previously, and set
\[L_{\mathbf{a}}(t,\v)=
\sum_{1\le i\le j\le H}a_i a_jL_{ij}(t\mathbf{a},\v).\]
Thus a rational solution $(t,\v)$ of $Q_{\mathbf{a}}(t,\v)=
L_{\mathbf{a}}(t,\v)=0$ produces a corresponding point
$[t\mathbf{a},\v]$ on $X$. 
We now have the following result, which plays a similar role to 
Lemma \ref{good}, but is much easier to prove.

\begin{lemma}\label{good'}
In the situation of Lemma \ref{splin},
assume that at least one linear form $L_{ij}$ depends explicitly
on $\mathbf{v}$.
Then there is a non-empty
Zariski-open set of
$[\mathbf{a}] \in \PP^{H-1}(\QQ)$ such that the equations 
$Q_{\mathbf{a}}(t,\v)=L_{\mathbf{a}}(t,\v)=0$ have a real solution.
\end{lemma}

Let us now show how to proceed under the assumption of Lemma~\ref{good'}.
The equations
\[Q_{\mathbf{a}}(t,\v)=L_{\mathbf{a}}(t,\v)=0\]
describe the intersection of a quadric hypersurface with a
hyperplane.  In general such an intersection will have a rational
point whenever there is a real point, as long as
$\rank(Q_{\mathbf{a}})\ge 6$. However
\[Q_{\mathbf{a}}(t,\v)=R(\mathbf{a})t^2+S(\v)\]
with new quadratic forms $R$ and $S$,
where as before we may assume that $\rank(S)=n-H$.  Thus 
\beql{rlb}
\rank(Q_{\mathbf{a}})\ge\rank(S)=n-H=n-(h+4)\ge 9,
\eeq
which suffices to complete the proof of Theorem \ref{small_h}.

It remains to  consider the possibility that the
 assumption is not met in Lemma~\ref{good'}.
Thus  all of the linear forms $L_{ij}$ are independent of
$\mathbf{v}$ and so
$C(\mathbf{x})=C(\mathbf{u})$. 
It is  enough to find a non-trivial rational zero of the
system $C(\mathbf{u})=0$ and
\[
  Q_{\mathbf{u}}(1,\v)=R(\mathbf{u})+S(\v) = 0,
\]
where, as  above, $\rank(S)\geq 9$.  
Since $X$ is
absolutely irreducible the same is true for 
$C$.
Likewise, on making a suitable linear change of variables, 
we may assume that $C$ is a non-degenerate cubic form
in $H'\leq H$ variables.  
We will also assume, temporarily,  that the locus 
of rational solutions to $C=0$ is dense in the locus of real 
solutions. We call this the 
 ``real density hypothesis'', for convenience. 

If $S$ is indefinite or is singular, then it suffices to take 
$\mathbf{u}=\mathbf{0}$ and to solve $S(\mathbf{v})=0$ non-trivially 
over the rationals. This will certainly be possible, since $S$ has
rank at least 9.  We therefore suppose $S$ is definite, and without 
loss of generality we take $S$ to be positive definite.

We now make use 
of our assumption that there is a non-singular real zero of the system
$C=Q=0$ under consideration. The variety $X$ cannot be contained in
$Q=R=0$, since the latter will be irreducible of degree 4.  It
therefore follows from Lemma \ref{lem:density} that $X$ has a real
point $(\mathbf{u}_0,\mathbf{v}_0)$ with $R(\mathbf{u}_0)\not=0$.  Our
assumption that $S$ is positive definite then shows that we must have
$R(\mathbf{u}_0)<0$. 

In view of the real density hypothesis we 
can now find a rational zero $\mathbf{u}$ of $C$ sufficiently close to 
$\mathbf{u}_0$ that $R(\mathbf{u})<0$. Then, since
$\rank(S)\geq 9$, there will be a rational vector
$\v$ such that $S(\v)=-R(\u)$. This produces a non-trivial rational point
$[\x]=[(\u,\v)]$ on $X$, thereby completing the proof
in the second case, subject to the real density hypothesis. 

Finally we claim that the real density hypothesis holds if 
$\qorder(C)\geq 17$.  
If $h\geq 14$,  a straightforward modification of the main result 
in \cite{14} establishes the desired conclusion (cf.\ \cite[Lemma~1]{mike'}). 
Alternatively, if $h\leq 13$,  then it follows from our lower bound for 
$\qorder(C)$ that 
 \[H'\geq\qorder(C)\geq 17\geq h+4.\]
But then the claim follows from 
work of Swarbrick Jones \cite[Lemma~2]{mike'}.

\section{Proof of Lemmas \ref{smoothp}--\ref{good'}}\label{s:11}

It remains to establish Lemmas \ref{smoothp} to \ref{good'}, and we
begin with the first of these. For the proof we work over $\QQ_p$. The
quadratic form $Q$ may be written as a non-singular form in variables 
$x_1,\ldots,x_{\rho}$, and vanishes on a linear space of projective
dimension at least $\lceil(\rho-6)/2\rceil\ge 9$, in terms of 
these variables.  Hence,
as remarked in the introduction in
connection with Theorem \ref{main}, the form
$C$ will vanish at a $p$-adic point $P$, which we see may be taken to
be a non-singular point on $Q=0$. If we choose coordinates so 
that $P=[1,0,\ldots,0]$ 
our forms take the shape
\[C(\x)=x_1^2L_1(x_2,\ldots,x_n)+x_1Q_1(x_2,\ldots,x_n)+C_1(x_2,\ldots,x_n)\]
and
\[Q(\x)=x_1L_2(x_2,\ldots,x_n)+Q_2(x_2,\ldots,x_n).\]
Then $L_2$ cannot vanish identically, since $P$ is a non-singular point
on $Q=0$.  Moreover, if $L_1$ and $L_2$ are not proportional then $P$
is a smooth point on $X$.  We may therefore assume that $L_1=cL_2$.
Thus if $C'=C+LQ=C-cx_1Q$, we can write $C'(\x)$ in 
the simpler shape
\[C'(\x)=x_1Q_1(x_2,\ldots,x_n)+C_1(x_2,\ldots,x_n).\] 
Since $L_2$ does not vanish identically we can make a change of
variables to replace $L_2$ by $x_2$, say, so that $Q(\x)$ becomes
\begin{align*}
Q(\x)&=x_1x_2+Q_2(x_2,\ldots,x_n)\\
&=x_1x_2+x_2L_3(x_2,\ldots,x_n)+
Q_3(x_3,\ldots,x_n),
\end{align*}
say.  Now replacing $x_1$ by $x_1+L_3$ we further simplify $Q$ to the
shape $x_1x_2+Q_3(x_3,\ldots,x_n)$.  We then write
\[Q_1(x_2,\ldots,x_n)=x_2L_4(x_2,\ldots,x_n)+Q_4(x_3,\ldots,x_n)\]
and replace $C'$ by $C'-L_4Q$ so that (renaming our forms) 
\begin{align*}
C'(\x)&=x_1Q_1(x_3,\ldots,x_n)+C_1(x_2,\ldots,x_n)\\ 
Q(\x)&=x_1x_2+Q_2(x_3,\ldots,x_n).
\end{align*}

Consider the projection $X\rightarrow \PP^{n-2}$ from the point 
$[1,0,\ldots,0]$.  The Zariski-closure of the image of this rational map is 
the hypersurface 
\[Y:\quad
x_2C_1(x_2,\ldots,x_n)-Q_1(x_3,\ldots,x_n)Q_2(x_3,\ldots,x_n)=0\]
in $\PP^{n-2}$.  In fact $X$ and $Y$ are birational to each 
other over $\QQ$, the reverse map being given by 
$$
[x_2,\ldots,x_n]\mapsto 
\begin{cases}
[-Q_2/x_2,x_2,\ldots,x_n], & \mbox{if $x_2\neq 0$,}\\
[-C_1/Q_1,x_2,\ldots,x_n], & \mbox{if $Q_1\neq 0$,}
\end{cases}
$$
on the Zariski-open subset where $(x_2,Q_1)\neq (0,0)$.
Lemma \ref{lem:irred} ensures that $X$ is absolutely irreducible, 
and we therefore deduce that $Y$ is also absolutely irreducible. 
Lemma \ref{smoothp} will follow if we are able to show that the
$p$-adic points on $X$ are Zariski-dense. For this it will  suffice  
to show that the $p$-adic points on $Y$ are Zariski-dense. 
This will follow from   Lemma
\ref{lem:density} if we can show that $Y$ has a non-singular $p$-adic point. 

To verify the existence of a non-singular $p$-adic point on $Y$, 
we consider points with $x_2=Q_2=0$.  Such a point
will be non-singular on $Y$ provided that $\nabla Q_2\not=\mathbf{0}$
and that $Q_1$ and $C_1$ are not both zero. However 
\[\rank(Q_2)\ge \rank(Q)-2=\rho-2\ge 21 \ge 5\] 
so that the $p$-adic zeros
$[x_3,\ldots,x_n]$ 
of $Q_2$ are Zariski-dense on $Q_2=0$. In particular we can
choose a point where $\nabla Q_2\not=\mathbf{0}$, and where
$Q_1$ and $C_1$ are not both zero, unless both $Q_1$ and $C_1$ are
multiples of $Q_2$. 
However if $Q_2$ divides $C'$ we have $C'=L'Q_2$ for some linear form
$L'$, and hence
\[C=L''Q+C'=L''Q+L'Q_2=(L''+L')Q-L'x_1x_2=\overline{L}Q+L_1L_2L_3,\]
say.  Here $L_1,L_2,L_3$ and $\overline{L}$ are linear forms defined
over $\QQ_p$. If $\overline{L}$ were defined over $\QQ$ then we would
have $\qorder(C)=\qorder(L_1L_2L_3)\le 3$, contrary to our
hypotheses. Thus there is a field automorphism $\sigma$ say, such that
$\overline{L}^{\sigma}\not=\overline{L}$. Since $C^{\sigma}=C$ this yields
\[(\overline{L}^{\sigma}-\overline{L})Q=
L_1^{\sigma}L_2^{\sigma}L_3^{\sigma}-L_1L_2L_3.\]
Changing variables we may write
$\overline{L}^{\sigma}-\overline{L}=x_1$, whence $x_1Q$ has order at
most 6. We claim in general that for any form $F(x_1,\ldots,x_n)$, the
order of $F$ is at most one more than the order of
$x_1F(x_1,\ldots,x_n)$.  Given this claim we would deduce that
$\rank(Q)\le 7$, contrary to hypothesis.  Thus to complete the 
proof of Lemma \ref{smoothp} it is enough to establish the
claim. However this is easy, since if we can write 
\[x_1F(x_1,\ldots,x_n)=G(L_1,\ldots,L_m)\]
with forms 
\[L_i(x_1,\ldots,x_n)=a_ix_1+\overline{L_i}(x_2,\ldots,x_n)\]
then $G(\overline{L_1},\ldots,\overline{L_m})$ must vanish
identically, and $F$ will be a function of $x_1$ and $\overline{L_1},
\ldots,\overline{L_m}$.  This suffices for the claim.
\bigskip

The next result to prove is Lemma \ref{SD}.  Theorem A of 
Colliot-Th\'el\`ene, Sansuc and Swinnerton-Dyer \cite{CTSSD} tells us
that an absolutely irreducible non-degenerate intersection of quadrics 
in $m\ge 9$ variables satisfies the smooth Hasse principle.  
Of course, if the intersection
is degenerate there will trivially be a rational point (though not
necessarily a smooth rational point).  Thus we may assume that our
intersection is non-degenerate. We claim that $\rank(h)\ge 5$ for
every form $h$ in the pencil generated by $f$ and $g$, either
over $\overline{\QQ}$, or over some $\QQ_p$. This follows from our
hypotheses if $h$ is proportional to a rational form.  Otherwise there
is some field automorphism $\sigma$ such that $h^{\sigma}$ and $h$ are
not proportional.  However $h^{\sigma}$ is also in the pencil
generated by $f$ and $g$.  Now if $\rank(h)\le 4$ then
$\rank(h^{\sigma})\le 4$ so that the variety $h^{\sigma}=h=0$ would be
degenerate. This however is impossible given our previous assumption,
since $h^{\sigma}$ and $h$ generate the same pencil as $f$ and $g$.
Our claim is therefore established.
In particular we now see that the
intersection $f=g=0$ will be absolutely irreducible, by 
\cite[Lemma 1.11]{CTSSD}, so that the Hasse principle applies.

The variety $f=g=0$ has a smooth real 
point by hypothesis, and we claim that there are smooth $p$-adic 
points for every prime $p$. This will suffice for the proof of the lemma.

To prove this we note that for any prime $p$ there is a 
$p$-adic point by the result of Demyanov
\cite{Dem}, since $m\ge 9$.  Clearly we may assume that this point is a singular
point, since otherwise the claim is immediate.  Then, choosing
coordinates so that the point in question is at $[1,0,\ldots,0]$, 
the forms become $x_1L_1(x_2,\ldots,x_m)+f_1(x_2,\ldots,x_m)$ and 
$x_1L_2(x_2,\ldots,x_m)+g_1(x_2,\ldots,x_m)$. Here the forms $L_1$ and
$L_2$ cannot both vanish since we are assuming that $f=g=0$ is
non-degenerate.  Moreover they must be proportional since
$[1,0,\ldots,0]$ was assumed to be singular. Thus, after replacing the
forms $f$ and $g$ by a suitable linear combination, and after making a
further change of variables, we may take $L_2$, say, to vanish, and
take $L_1=x_2$. Now, since $\rank(g_1)\ge 5$ by what we proved above, 
we see that $g_1=0$ has a smooth $p$-adic zero.  Its smooth
$p$-adic zeros are therefore Zariski-dense.  Choosing such a zero with
$x_2\not=0$ we may then set $x_1=-x_2^{-1}f_1(x_2,\ldots,x_m)$,
obtaining a smooth point on $f=g=0$.  This establishes Lemma \ref{SD}.
\bigskip

We turn now to the proof of Lemma \ref{good}. If it were the case that
for every $\mathbf{a}\in\RR^h$ there is a linear combination
$C_{\mathbf{a}}(t,\v)+\lambda Q_{\mathbf{a}}(t,\v)$ with rank at most 1, then it would
be impossible for the variety $C_{\mathbf{a}}(t,\v)=Q_{\mathbf{a}}(t,\v)=0$ to have a
non-singular zero. We therefore begin by showing that this case cannot
arise.
\begin{lemma}\label{r>1}
Suppose that $n-h\ge 5$, and that 
$$
\qorder(C)\ge\max(h+1,3).
$$
Then either $X(\QQ)\neq \emptyset$,  or there is at least one non-zero 
$\mathbf{a}\in\QQ^h$
such that every linear combination $C_{\mathbf{a}}(t,\v)+\lambda Q_{\mathbf{a}}(t,\v)$
with $\lambda\in \overline{\QQ}$
has rank $2$ or more.
\end{lemma}

\begin{proof}
For the proof we write $Q(\u,\v)=R(\u)+S(\v)$ as before, with $\rank(S)=n-h$. 
We will assume for a contradiction that for every rational $\mathbf{a}$ there
is some $\lambda$ for which $C_{\mathbf{a}}(t,\v)+\lambda Q_{\mathbf{a}}(t,\v)$ has
rank at most 1.  In particular, for any $j$ between 1 and $h$ we may
define $\mathbf{a}$ by taking $a_i=0$ for $i\not=j$ and $a_j=1$. Then, setting
$t=0$, we see that $B_j(\v)+\lambda_j S(\v)$ has rank at most 1, in
the notation \eqref{ABRS}. In the same way, for distinct positive
integers $j,k\le h$, we may
take $a_i=0$ for $i\not=j,k$ and $a_j=a_k=1$, finding that 
$B_j(\v)+B_k(\v)+\lambda_{j,k} S(\v)$ has rank at most 1. This
produces equations
\[B_j(\v)+\lambda_j S(\v)=L_j(\v)^2,\;\;\;
B_k(\v)+\lambda_k S(\v)=L_k(\v)^2\]
and
\[B_j(\v)+B_k(\v)+\lambda_{j,k} S(\v)=L_{j,k}(\v)^2.\]
Here the coefficients $\lambda$ and the linear forms $L$ are defined
over $\overline{\QQ}$. By subtraction we find that either
$\lambda_j+\lambda_k=\lambda_{j,k}$, or that $\rank(S)\le 3$. Since we
have assumed that $\rank(S)=n-h\ge 5$ we deduce that
$\lambda_j+\lambda_k=\lambda_{j,k}$, and then that
$L_j^2+L_k^2=L_{j,k}^2$. This can happen only when $L_j,L_k$ and
$L_{j,k}$ are proportional, allowing us to conclude that there is a
non-zero linear form $L_0$ defined over $\overline{\QQ}$, and constants 
$\mu_j\in\overline{\QQ}$, such that 
\[ B_j(\v)+\lambda_j S(\v)=\mu_j L_0(\v)^2 \] 
for every $j$. In
fact, if $\lambda_j\not\in\QQ$ we can apply some nontrivial Galois
automorphism $\sigma$ to show that 
$B_j(\v)+\lambda_j^\sigma S(\v)=\mu_j^\sigma (L_0(\v)^\sigma)^2$.
Then by subtraction we see that $(\lambda_j-\lambda_j^\sigma)S(\v)$ 
has
rank at most 2, again contradicting our assumptions. Thus all the
$\lambda_j$ are in $\QQ$, so that we may suppose $L_0$ and the $\mu_j$ to
be defined over $\QQ$.

Taking 
\[L(\x)=\sum_{i=1}^h \lambda_i u_i\]
we now replace $C(\x)$ by $C'=C(\x)+L(\x)Q(\x)$.  This new cubic may be
written in the shape given by \eqref{ABRS*}, with a different function 
$A(\u)$, and with $B_i(\v)$ replaced by $B_i'(\v)=B_i(\v)+\lambda_i 
S(\v)=\mu_i L_0(\v)^2$. In particular we will have $h(C')\le h$, and
since we chose our original cubic $C$ to have $h(C)=h_Q(C)$ we see
in fact that $h(C')=h_Q(C)=h$.  For ease of notation we will just
write $C$ in place of $C'$ henceforth, and assume that 
\beql{extra} 
B_i(\v)=\mu_i L_0(\v)^2. 
\eeq

Now suppose that 
\beql{CQJ}
C_{\mathbf{a}}(t,\v)+\lambda Q_{\mathbf{a}}(t,\v)=(\al t+J(\v))^2
\eeq
for some $\al$ and $J(\v)$ defined over $\overline{\QQ}$. Then, on comparing the
terms not involving $t$, and using \eqref{extra}, we see that 
\begin{align*}
J(\v)^2&=\sum_{j=1}^h a_j B_j(\v)+\lambda S(\v)\\
&=\left(\sum_{j=1}^h \mu_j a_j\right)
L_0(\v)^2+\lambda S(\v).
\end{align*}
Using the fact that $\rank(S)\ge 5$ once again we conclude that
$\lambda=0$ and that
$J(\v)$ is proportional to $L_0(\v)$, and hence 
equal to $\beta L_0(\x)$ say.

We now expand \eqref{CQJ} further, using \eqref{ABRS}. We then
see from the linear term in $t$ that
\beql{DL0}
\sum_{j=1}^s D_j(\mathbf{a})v_j=2\al\beta L_0(\v).
\eeq
Thus for every rational vector $\mathbf{a}$ the linear form $\sum_j
D_j(\mathbf{a})v_j$ is proportional to $L_0(\v)$. This can happen only when
the quadratic forms $D_j$ are all proportional to each other, of the
shape $\nu_j D(\mathbf{a})$ say, with constants $\nu_j\in\QQ$.  This allows
us to write
\[C(\u,\v)=A(\u)+D(\u)L'(\v)+\ell(\u)L_0(\v)^2\]
for suitable linear forms $L'$ and $\ell$ defined over $\QQ$, and indeed
\eqref{DL0} shows that we may take $L'(\v)=L_0(\v)$.

It follows that
$C_{\mathbf{a}}(t,\v)=A(\mathbf{a})t^2+D(\mathbf{a})tL_0(\v)+
\ell(\mathbf{a})L_0(\v)^2$, which must
have rank at most one for every choice of $\mathbf{a}\in\QQ^h$.  
If $L_0$ vanishes identically, or if $\ell(\u)$ and
$D(\u)$ both vanish identically, then $C(\x)=A(\u)$, which has order 
at most $h$, contrary to the hypothesis of Lemma \ref{r>1}.
Thus
$D(\mathbf{a})^2=4A(\mathbf{a})\ell(\mathbf{a})$ for 
any $\mathbf{a}\in\QQ^h$, and then
$D(\u)=2\ell(\u)\ell'(\u)$ and $A(\u)=\ell(\u)\ell'(\u)^2$ for
some linear form $\ell'(\u)$ defined over $\QQ$. 
However in this case 
\[C(\x)=A(\u)+D(\u)L_0(\v)+\ell(\u)L_0(\v)^2=\ell(\u)\{\ell'(\u)+L_0(\v)\}^2,\]
which has order at most 2, again contradicting our hypotheses.  This
therefore establishes the lemma.
\end{proof}

The next stage in the proof of Lemma \ref{good} is the
following result.
\begin{lemma}\label{goodQb}
Under the hypotheses of Lemma \ref{r>1}, either 
$X(\QQ)\neq \emptyset$, or there is at least one non-zero
$\mathbf{a}\in\QQ^h$ such that the variety
$$C_{\mathbf{a}}=Q_{\mathbf{a}}=0$$ 
has a point
$(t,\v)\in\overline{\QQ}^{\, 1+s}$ with $t\not=0$, at which
$\nabla C_{\mathbf{a}}$ and $\nabla Q_{\mathbf{a}}$ are not proportional.
\end{lemma}

\begin{proof}
By Lemma \ref{r>1} we may choose $\mathbf{a}$ so that every form in the pencil
generated by $C_{\mathbf{a}}$ and $Q_{\mathbf{a}}$ has rank at least 2. As before we
may assume that $\rank(S)=n-h\ge 5$, whence $\rank(Q_{\mathbf{a}})\ge 5$.  We
will show in general that if $A(\y)$ and $B(\y)$ are quadratic forms
such that $\rank(A)\ge 5$, and such that every form in the pencil
generated by $A$ and $B$ over $\overline{\QQ}$ has rank at least 2,
then $A=B=0$ has a point with $\nabla A$ not proportional to $\nabla
B$, and lying off any given hyperplane $L(\y)=0$. (In this general
formulation the condition
$t\not=0$ corresponds to a requirement of the type
$L(y_1,\ldots,y_n)\not=0$.)
Without loss of generality we can take $B$ with as small rank, $r$
say, as possible.  If $r\ge 3$ then the variety $A=B=0$ is irreducible
of degree 4 and codimension 2, and 
is not contained in
the hyperplane $L=0$. Since the variety $A=B=0$ has projective
dimension $n-3\ge n-h-3\ge 2$ there will be a non-empty Zariski-open 
set of points satisfying the conditions of the lemma.

We therefore assume that $B$ has rank exactly 2, and write $B=x_1x_2$.
Since $L$ cannot be proportional to both $x_1$ and $x_2$ we may assume
that $x_1$, say, is not proportional to $L$. We set $x_1=0$ and
$L'(x_2,\ldots,x_n)=L(0,x_2,\ldots,x_n)$, and look
for points on $A=x_1=0$ with $x_2L'\not=0$ and such that $\nabla A$ is
not proportional to $(1,0,\ldots,0)$. However
$A'=A(0,x_2,x_3,\ldots,x_n)$ has rank at least $\rank(A)-2\ge 3$ and
hence is an absolutely irreducible quadratic form.  Moreover at least
one partial derivative $P_i=\partial A'/\partial x_i$ for $i=2,\ldots,n$
is not identically zero. Thus $A'$ cannot divide $x_2L'P_i$, whence
$A'=0$ has a point at which $x_2L'P_i\not=0$. This produces a point
$(0,x_2,\ldots,x_n)$ on $A=B=0$ for which $L\not=0$ and such that
$\nabla A$ is not proportional to $\nabla B$. This completes the
proof of the lemma.
\end{proof}

We are now ready to complete the proof of Lemma \ref{good}.  The
variety $X\subset\PP^{n-1}$ is defined by $C(\u,\v)=Q(\u,\v)=0$ and is
absolutely irreducible, by Lemma \ref{lem:irred}.  
The points $[\u,\v]$ on $X$ for which
$[t,\v]=[1,\v]$ is a singular point of $C_{\u}(t,\v)=Q_{\u}(t,\v)=0$
form a Zariski-closed subset of $X$, and by Lemma \ref{goodQb} it is a
proper subset of $X$. We have assumed that $X$ has a smooth real
point, and by Lemma \ref{lem:density} the real points must be Zariski-dense on
$X$. Hence there is a Zariski-dense set of smooth real points $[\u,\v]$ of $X$, 
with $\u\not=\mathbf{0}$
and such that $[1,\v]$ is a smooth point of
$C_{\u}(t,\v)=Q_{\u}(t,\v)=0$.  It follows in particular that there is
a non-zero real $\u$ such that
$C_{\u}(t,\v)=Q_{\u}(t,\v)=0$ has a smooth real point $[1,\v]$. Suppose now
that $\mathbf{a}_m$ is a sequence of rational points tending to $\u$
in the real metric.  Write $A(t,\v)$ and $B(t,\v)$ for the quadratic forms 
$C_{\u}(t,\v)$ and $Q_{\u}(t,\v)$, and write $A_m,B_m$ for the
corresponding forms when $\u$ is replaced by $\mathbf{a}_m$.  Then
$A_m$ and $B_m$ tend to $A$ and $B$ respectively.  However $A$ and $B$
have a smooth real zero at $[1,\v]$, whence it follows that $A_m$ and
$B_m$ will also have a smooth real zero $[1,\v_m]$, say, if $m$ is
large enough.  This suffices for the proof of Lemma \ref{good}. In
particular the rational points $[\mathbf{a}]\in\PP^{h-1}$ obtained in this way
cannot be restricted to a proper subvariety of $\PP^{h-1}$, since the 
points $[\u]$ were Zariski-dense. 
\bigskip

Moving on to Lemma \ref{splin}, we begin by observing that if
$\alpha C_{\mathbf{a}}+\beta Q_{\mathbf{a}}$ has rank at most 4 then, on
setting $t=0$, we must have 
$$
\rank\left(\alpha\sum_{i=1}^ha_iB_i(\v)+\beta S(\v)\right)\le 4.
$$
Since $\rank(S)=n-h\ge 13$ we will have $\alpha\not=0$, and we may
therefore assume that $\alpha=1$.  We now consider the variety
\[\mathcal{I}=\left\{[u_1,\ldots,u_h,\beta]\in\PP^h:
\rank\left(\sum_{i=1}^hu_iB_i(\v)+\beta S(\v)\right)\le 4\right\}.\] 
The projection $[u_1,\ldots,u_h,\beta]\mapsto [u_1,\ldots,u_h]$ is
well-defined on $\mathcal{I}$ since
$[0,\ldots,0,1]\not\in\mathcal{I}$.  Its image is Zariski-dense in
$\PP^{h-1}$ and must therefore be the whole of $\PP^{h-1}$, so that
for every $[\mathbf{a}]\in \PP^{h-1}$, 
there is  a corresponding $\beta$ such that
\beql{le4'}
\rank\left(\sum_{i=1}^ha_iB_i(\v)+\beta S(\v)\right)\le 4.
\eeq
It is possible indeed that
this might still be true with the bound 4 replaced by some smaller
number.  We therefore define $\tau\le 4$ as the smallest integer such
that \eqref{le4'} is solvable for $\beta$, for all $\mathbf{a}$.

We now claim that, after replacing $C$ by $C+LQ$ for a suitable
linear form $L=L(\u)$ defined over $\QQ$, and after making a suitable linear
change of variables among the $u_i$, we will have $\rank(B_i)=\tau$
for $1\le i\le h$.  Moreover it will remain true that for every
$\mathbf{a}$ there is a corresponding $\beta=\beta(\mathbf{a})$ with
\[\rank\left(\sum_{i=1}^ha_iB_i(\v)+\beta S(\v)\right)\le\tau.\]

To establish the claim we first note that there is a 
Zariski-dense set of values of $[\mathbf{a}]$ 
such that the rank given above is actually
equal to $\tau$. Thus we may choose a linearly independent set of
vectors $\mathbf{a}_1,\ldots,\mathbf{a}_h\in\QQ^h$ with this
property.  Then, after a suitable change of variable among the $u_i$
we can suppose that
\[\rank\left(B_i(\v)+\beta_i S(\v)\right)=\tau,\;\;\;(1\le i\le h).\]
If $\beta_i$ were irrational for some $i$ there would be a 
Galois automorphism $\sigma$ such that 
$\beta_i^{\sigma}\not=\beta_i$.  We would then have
\[\rank\left(B_i(\v)+\beta_i^{\sigma} S(\v)\right)=\tau,\]
whence $\rank\big((\beta_i^{\sigma}-\beta_i)S(\v)\big)\le 2\tau$,
by subtraction.  This however is impossible since
$\beta_i^{\sigma}-\beta_i\not=0$ and $\rank(S)=n-h\ge 13$. Thus all
the $\beta_i$ must be rational.  We then define
\[L(\u)=\sum_{i=1}^h \beta_i u_i\]
and consider $C'=C+LQ$. The corresponding quadratic forms $B_i'(\v)$
are now $B_i(\v)+\beta_i S(\v)$, 
and therefore have rank $\tau$. The
claim then follows.

To complete the argument we take any index $i=2,\ldots,h$, and any
$\mu\in\QQ$.  There is
then a $\gamma_i\in\overline{\QQ}$ such that
\[\rank\left(B_1(\v)+\mu B_i(\v)+\gamma_i S(\v)\right)\le\tau.\]
However $B_1$ and $B_i$ both have rank 
$\tau$ so that
$$\rank(\gamma_i S(\v))\le 3\tau\le 12,
$$ 
by subtraction.  Since
$\rank(S)=n-h\ge 13$ this would give a contradiction unless
$\gamma_i=0$, as we now assume. It therefore follows that $\rank(B_1+\mu B_i)\le
\tau$ for every $i$, and for every choice of $\mu$.

We proceed to make a change of variables among the $v_j$ so as to make
$B_1(\v)=B_1^*(v_1,\ldots,v_{\tau})$.  We now claim that
$B_i(0,\ldots,0,v_{\tau+1},\ldots,v_s)$ must vanish identically, for
every $i$. If this were not the case we could introduce a change of
variable among $v_{\tau+1},\ldots,v_s$ so as to make $v_{\tau+1}^2$
appear with coefficient 1, in $B_i$.  The
$(\tau+1)\times(\tau+1)$ minor of $B_1+\mu B_i$ corresponding to the
first $\tau+1$ rows and first $\tau+1$ columns would then be a
polynomial $P(\mu)$ say, with linear term $\mu\det(B_1^*)$.  
Since $\rank(B_1^*)=\tau$ we have $\det(B_1^*)\not=0$ so that 
$P(\mu)$ does not vanish identically.  Thus 
there
can be at most finitely many values of $\mu$ for which
$P(\mu)=0$. Taking any other value of $\mu$ produces a combination
$B_1+\mu B_i$ of rank strictly greater than $\tau$, which is a
contradiction. This establishes our claim. 

We therefore see that $B_i(0,\ldots,0,v_{\tau+1},\ldots,v_s)$
vanishes identically, for every $i$, so that each of the quadratic
forms $B_1,\ldots,B_h$ may be written in the shape
$B_i(\v)=v_1\ell_{i1}(\v)+\cdots+v_{\tau}\ell_{i\tau}(\v)$. Thus, if
we relabel $v_1,\ldots,v_{\tau}$ as $u_{h+1},\ldots u_{h+\tau}$ we
will be able to put $C(\x)$ into the form
\[C(\x)=C(\u,\v)=\sum_{1\le i\le j\le H}u_i u_jL_{ij}(\u,\v),\]
with $H=h+\tau$.  Lemma \ref{splin} then follows.
\bigskip

The proof of Lemma \ref{good'} is rather easy.
Since at least one linear form $L_{ij}(\mathbf{u}, \v)$ depends
explicitly on $\v$, we can choose $\mathbf{a} \in \QQ^H$ such that 
$L_\mathbf{a}(t, \v)$ also explicitly depends on $\v$.
In particular, the equation $L_\mathbf{a}(t, \v)=0$ has solutions
with $t \ne 0$, and they are Zariski-dense amongst the set of all
solutions. Hence, taking such a suitable $\mathbf{a}\in\QQ^H$,
we see that
$\rank(Q_{\mathbf{a}})\ge \rank(S)\ge 9$, as in 
\eqref{rlb}. It follows that the variety
$Q_{\mathbf{a}}(t,\v)=L_{\mathbf{a}}(t,\v)=0$ will have a point
of the form $[1, \v]$
over  $\overline{\QQ}$ which is non-singular in the sense that 
$\nabla Q_{\mathbf{a}}$ is not proportional to $\nabla L_{\mathbf{a}}$.

We now argue as in the final stages of the proof of Lemma \ref{good}.
We have shown that there is a
point $[\u,\v]$ on $X$ such that $[1,\v]$ is a smooth point on
$Q_{\mathbf{a}}(t,\v)=L_{\mathbf{a}}(t,\v)=0$.  There is therefore a
non-empty Zariski-open subset of such points $[\u,\v]$.
However the variety
$C=Q=0$ is absolutely irreducible, and has a smooth real point.  The
real points are therefore Zariski-dense, by Lemma \ref{lem:density}. We 
choose any such point 
with $\u \not=\mathbf{0}$, and such that
$[1,\v]$ is a smooth point on
$Q_{\mathbf{a}}(t,\v)=L_{\mathbf{a}}(t,\v)=0$. Then, taking rational
points $\mathbf{a}_m$ converging to $\u$ in the real topology, we may
complete the argument as before.


\begin{thebibliography}{9}

\bibitem{aznar} 
V.N. Aznar, 
On the Chern classes and the Euler characteristic for non-singular
complete intersections.  
{\em Proc. Amer. Math. Soc.} {\bf 78} (1980), 143--148. 

\bibitem{bb}
R.C. Baker, 
{\em Diophantine inequalities}. London Mathematical Society
Monographs, 
Oxford University Press, 1986.

\bibitem{birch}
B.J. Birch,
Forms in many variables.
{\em Proc. Roy. Soc. Ser. A} {\bf 265} (1961/62), 245--263.


\bibitem{41}
T.D. Browning and D.R. Heath-Brown,
 Rational points on quartic hypersurfaces.
{\em J. reine angew. Math.} {\bf 629} (2009), 37--88. 

\bibitem{BDLW} J. Br{\"u}dern, R. Dietmann, J.Y. Liu and 
 T.D. Wooley. A Birch--Goldbach Theorem,
{\em Arch. Math. (Basel)} {\bf 94} (2010), 53--58.


\bibitem{CTSSD} J.-L. Colliot-Th\'{e}l\`{e}ne,  J. Sansuc and
P. Swinnerton-Dyer.
Intersection of two quadrics and Ch\^{a}telet surfaces. I.
{\em J. reine angew. Math.} {\bf 373} (1987), 37--107.

\bibitem{DL}
H. Davenport and D.J. Lewis, 
Non-homogeneous cubic equations.
{\em J. London Math. Soc.} {\bf 39} (1964), 657--671. 

\bibitem{deligne}P. Deligne, La conjecture de Weil. I. {\em
Inst. Hautes \'{E}tudes Sci. Publ. Math.}  {\bf 43} (1974), 273--307. 

\bibitem{Dem}V.B. Demyanov, 
Pairs of quadratic forms over a complete field with discrete norm 
with a finite field of residue classes. {\em
Izv. Akad. Nauk SSSR. Ser. Mat.} {\bf 20} (1956), 307--324. 

\bibitem{fulton} 
W. Fulton, {\em Intersection Theory}. 2nd ed.,
Springer-Verlag, 1998. 

\bibitem{harris} 
J. Harris, {\em Algebraic Geometry}. GTM {\bf 133}, Springer-Verlag, 1992.

\bibitem{hart}
R. Hartshorne, {\em Algebraic Geometry}.   GTM {\bf 52}, 
Springer-Verlag, 1977. 

\bibitem{hb-10} 
D.R. Heath-Brown,
Cubic forms in ten variables.
{\em Proc. London. Math. Soc.}
{\bf 47} (1983), 225--257.

\bibitem{HB-canada} 
D.R. Heath-Brown,
A multiple exponential sum to modulus $p^2$.
{\em Canad. Math. Bull.} {\bf 28} (1985), 394--396. 

\bibitem{HB-crelle} 
D.R. Heath-Brown,
A new form of the circle method
and its application to quadratic forms.
{\em J. reine angew. Math.}
{\bf 481} (1996), 149--206.

\bibitem{14}D.R. Heath-Brown, Cubic forms in 14 variables. 
{\em Invent. Math.} {\bf 170} (2007), 199--230. 

\bibitem{kollar}
J. Koll\'ar, Rationally connected varieties over local fields. 
{\em Annals
  of Math.} {\bf 150} (1999), 357--367.

\bibitem{leep}
D. Leep, Systems of quadratic forms.
{\em J. reine angew. Math.}
{\bf 350} (1984), 109--116.


\bibitem{lewis}
D.J. Lewis,
Cubic homogeneous polynomials over $\mathfrak{p}$-adic number fields.
{\em Annals of Math.} {\bf 56} (1952), 473--478.

\bibitem{mordell}
L.J. Mordell, A remark on indeterminate equations in several
variables.
{\em J. London Math. Soc.} {\bf 12} (1937), 127--129.


\bibitem{P} P.A.B. Pleasants, Cubic polynomials over algebraic
number fields. {\em J. Number Theory} {\bf 7} (1975), 310--344.

\bibitem{schmidt} W. Schmidt, The density of integer points on
  homogeneous varieties. {\em Acta Math.} {\bf 154} (1985),  243--296.


\bibitem{serre} 
J.-P. Serre, 
{\em Lie Algebras and Lie
groups}, Springer  LNM {\bf 1500}, 
2nd ed., Springer-Verlag,  1992.


\bibitem{mike'} 
M. Swarbrick Jones, 
Weak approximation for cubic hypersurfaces of large dimension,
{\em Algebra \& Number Theory}, to appear. (\texttt{arXiv:1111.4082})


\bibitem{wooley'}
T.D. Wooley, On simultaneous additive equations,  II.
{\em J. reine angew. Math.} {\bf 419} (1991), 141--198. 

\bibitem{wooley}
T.D. Wooley, On simultaneous additive equations,  IV.  {\em
  Mathematika} {\bf 45} (1998), 319--335. 

\bibitem{zahid}
J. Zahid, 
Simultaneous zeros of a cubic and quadratic form. 
{\em J. London Math. Soc.} {\bf 84 } (2011), 612--630. 

\end{thebibliography}
\end{document}